\documentclass[smallextended,referee,envcountsect]{svjour3}

\usepackage{float, enumerate, graphicx, amssymb, amsmath , color, longtable, multirow, lineno,url}
\usepackage{graphicx}
\usepackage[justification=centering]{caption}
\usepackage{amsmath}

\usepackage[active]{srcltx}
\usepackage{pifont}
\usepackage{paralist} 
\usepackage{amssymb}
\usepackage{authblk}
\usepackage{multirow,booktabs}
\usepackage{graphicx}
\graphicspath{{figures/}}
\usepackage{subfigure}
\usepackage{epsfig}
\usepackage{epstopdf}
\usepackage{longtable,booktabs}
\newtheorem{assumption}{Assumption}
\numberwithin{assumption}{section}
\numberwithin{equation}{section}
\newtheorem{rem}{Remark}
\usepackage{algorithm}
\usepackage{algorithmic}
\usepackage[a4paper,left=2cm,right=2cm,bottom=2.3cm,top=2.3cm]{geometry}
\usepackage[colorlinks,linkcolor=blue,citecolor=blue,                  bookmarksnumbered=true,
bookmarksopen=true,
colorlinks ]{hyperref}
\smartqed
\usepackage{graphicx}
\journalname{  }
\begin{document}
	
	\title{ A  new      subspace minimization conjugate  gradient method   for unconstrained  minimization   }
\author{ Zexian Liu$ ^{1} $ \and Yan Ni$ ^{1, *} $    \and  Hongwei Liu$ ^{2} $   \and Wumei Sun$ ^{2} $  }
\institute{
Yan Ni(\ding{41}),  e-mail:  ny20222202@163.com;	Zexian Liu,   e-mail: liuzexian2008@163.com
	\at  School of Mathematics and Statistics, Guizhou University, Guiyang, 550025, China
	\and
 Hongwei Liu, e-mail:hwliu@mail.xidian.edu.cn; Wumei Sun, sunwumei1992@126.com 
	\at   School of Mathematics and Statistics, Xidian University, Xi'an, 710126,   China
}
\date{Received: date / Accepted: date}
	
	\maketitle
	
	\begin{abstract}Subspace minimization conjugate gradient (SMCG) methods, as the generalization of traditional  conjugate gradient  methods,
		have become a class of quite efficient iterative methods for unconstrained optimization  and
		have attracted extensive attention recently.
		Usually, the search directions of   SMCG methods  are generated by minimizing   approximate models with  the approximation matrix $ B_k $    of   the objective function at the current iterate over the subspace spanned by the current gradient $ g_k $ and the latest search direction. The  $  g_k^TB_kg_k $ must be estimated properly in the calculation of the  search directions, which is crucial to the   theoretical  properties and the  numerical performance of   SMCG methods.    It is   a great   challenge to estimate it properly. 
		 An alternative    solution for this problem  might be to  design a new subspace minimization conjugate gradient method independent of the parameter   $ \rho_k  \approx   g_k^TB_kg_k $. The projection technique has been used successfully to generate conjugate gradient   directions such as   Dai-Kou conjugate gradient direction  (SIAM J Optim	23(1), 296-320, 2013).   Motivated by the above two observations, in the paper  
		we present a new   subspace minimization conjugate gradient methods by   using a projection technique  	based on the  memoryless quasi-Newton method.  More specially, we project the search direction of the  memoryless quasi-Newton method into the  subspace spanned by the current gradient and the latest search direction and drive a new search direction,  which is proved to be    descent.   Remarkably, the proposed method without any line search   enjoys the finite termination property for two dimensional convex quadratic functions, which   is helpful for designing    algorithm.  An adaptive scaling factor  in the search direction is given based on the above   finite termination property. The proposed method does not need to determine the parameter $ \rho_k $    and  can be regarded as an extension of   Dai-Kou conjugate gradient method.  The global convergence  of the proposed method  for general nonlinear  functions   is analyzed under the standard assumptions. Numerical comparisons on the 147 test function from the CUTEst library indicate the proposed method is very promising.

	\end{abstract}
	\section{Introduction}
	\vspace {-0.2cm}
	
	We consider the following   unconstrained optimization problem:
	\begin{equation*}
		\mathop {\min }\limits_{x \in { \mathbb R^n}} f(x).
	\end{equation*}
	where $f$ is continuously differential and its gradient is denoted by $ g $. Due to the low memory  requirement, simple  form and nice  numerical effect, conjugate gradient methods are a class of efficient iterative methods for large scale unconstrained optimization. Conjugate gradient methods are of the following form
	\begin{equation}\label{eq:CGform}
		x_{k+1} =x_k + \alpha_k d_k,
	\end{equation}
	where $\alpha_k$ is the stepsize obtained by a line
	search and  $d_k$ is the search direction given by
	\begin{equation}\label{eq:CGdk}
		{d_k} = \left\{ {\begin{array}{*{20}{c}}
				{ - {g_k},\;\;\;\;\;\;\;\;\;\;\;\;\;\;\;\;\;\;\;\text{if}\;k = 0,}\\
				{ - {g_k} + {\beta _k}{d_{k - 1}},\;\;\;\;\text{if}\;k > 0.}
			\end{array}} \right.
		\end{equation}
		Here $\beta_k$ is often called conjugate   parameter.  In the case
		that $ f $ is a convex quadratic function and the exact line search is performed,
		$ \beta_k $ should be the same. For nonlinear functions, however, different    $ \beta_k $ result in different conjugate gradient methods and their properties can be significantly different.  Some  well-known formulae for $ \beta_k $ are called the
		Fletcher-Reeves (FR) \cite{Fletcher1964Function}, Hestenes-Stiefel (HS)   \cite{Hestenes1952Methods}, Polak-Ribi\`ere-Polyak (PRP)   \cite{Polak1969The,Polak1969Note}   and Dai-Yuan (DY)  \cite{Dai1999A}
		formulae, and are given by
		\[\beta _k^{FR} = \frac{{{{\left\| {{g_{k}}} \right\|}^2}}}{{{{\left\| {{g_{k-1  }}} \right\|}^2}}},\;\;\;\;\beta _k^{HS} = \frac{{g_{k }^T{y_{k-1  }}}}{{d_{k-1  }^T{y_{k-1  }}}},\;\;\;\beta _k^{PRP} = \frac{{g_{k}^T{y_{k-1 }}}}{{{{\left\| {{g_{k-1 }}} \right\|}^2}}},\;\;\;\beta _k^{DY} = \frac{{{{\left\| {{g_{k}}} \right\|}^2}}}{{d_{k-1  }^T{y_{k-1  }}}},\]
	 	where $\left\| \cdot \right\| $ denotes the    Euclidean norm. 
		
		By deleting the	  third term of the memoryless quasi-Newton search direction,  Hager and Zhang  \cite{Hager2005A} presented  a    famous  efficient conjugate gradient   method (CG$\_ $DESCENT, We also call it HZ CG algorithm  for short) with
		\begin{align}
			{\beta_k ^{HZ}} = \frac{{g_{k + 1}^T{y_k}}}{{d_k^T{y_k}}} - 2 \frac{{{{\left\| {{y_k}} \right\|}^2}}}{{d_k^T{y_k}}}\frac{{g_{k + 1}^T{d_k}}}{{d_k^T{y_k}}},  \label{eq:HagerZhangbeta}
		\end{align}
and established the global   convergence   under the standard Wolfe line search. And the numerical results in \cite{Hager2005A,HagerLMCGDESCENT} indicated that CG$\_ $DESCENT with
		the approximate Wolfe  line search (AWolfe line search):
		$$ \sigma g_k^T{d_k} \le g{\left( {{x_k} + {\alpha _k}{d_k}} \right)^T}{d_k} \le \left( {2\delta  - 1} \right)g_k^T{d_k}, $$  where $  0 < \delta  < 0.5  \; \text{and} \; \delta  \le \sigma  < 1, $   is very efficient. In 2013, Dai and Kou \cite{Dai2013A} projected a multiple of the memoryless BFGS direction of Perry \cite{Perry1977A} and Shanno \cite{Shanno1978On}   into the manifold
		$\left\{ { - {g_{k + 1}} + s{d_k}:s \in \mathbb R} \right\} $
		and presented a family of   conjugate gradient algorithms (CGOPT, We also call them Dai-Kou CG algorithms for short) with the
		improved Wolfe line search,
		and  the numerical results in \cite{Dai2013A} suggested that CGOPT with   the following  parameter:
		\begin{align}
			{\beta_k ^{DK}} = \frac{{g_{k + 1}^T{y_k}}}{{d_k^T{y_k}}} - \frac{{{{\left\| {{y_k}} \right\|}^2}}}{{d_k^T{y_k}}}\frac{{g_{k + 1}^T{d_k}}}{{d_k^T{y_k}}}   \label{eq:Daibeta}
		\end{align}
		is the most efficient.  CG$\_ $DESCENT and CGOPT are both popular and quite efficient CG software packages. 	So far, conjugate gradient methods have attracted extremely extensive attention and the advance about conjugate  gradient methods can be referred as \cite{Andrei2020Nonlinear}.

		In conjugate gradient methods, the stepsize $ \alpha_k $ is often required to satisfy certain line search conditions. Among them, the strong Wolfe 	line search is often used  in the early convergence analysis, which  aims to find a
		stepsize   satisfying the following conditions
		\begin{gather}\label{eq:StandWolf1}	f\left( {{x_k} + {\alpha _k}{d_k}} \right) \le f\left( {{x_k}} \right) + \sigma {\alpha _k}g_k^T{d_k},  \\ \label{eq:StrongWolf2} \left |  g_{k+1} ^T{d_k} \right |  \le - \delta g_k^T{d_k},	\end{gather}
		where $ 0< \delta < \sigma < 1. $
		The standard Wolfe line search is also preferred due to the relatively easy numerical implementation, which aims to find a stepsize   satisfying \eqref{eq:StandWolf1}  and
		\begin{equation}
			\label{eq:StandWolf2}    g_{k+1} ^T{d_k}    \ge   \delta g_k^T{d_k}.	\end{equation}
		 		The sufficient descent property of the search direction plays an important role in the convergence analysis, which requires the search direction to satisfy	
		\begin{equation}\label{eq:SuffDesc1}g_k^T{d_k} \le  - c {\left\| {{g_k}} \right\|^2},  \	\end{equation}
		where $ c>0 $.
		
	 The subspace minimization conjugate gradient (SMCG) methods are the  the generalization of traditional conjugate gradient methods, and have also received much attention recently. The subspace minimization conjugate gradient  methods  were first   proposed by Yuan and Stoer \cite{Yuan1995A} in 1995, where the search direction is computed    by minimizing a  quadratic    model over the    subspace   $V_k = Span\left\{ {{g_k},{s_{k - 1}}} \right\}  $:
		\begin{equation}\label{eq:QuaModel}\;\mathop {\min }\limits_{d \in V_k} \;g_k^Td + \frac{1}{2}{d^T}{B_k}d,\;	\end{equation}
		where $B_k \in \mathbb R^{n \times n } $ is a symmetric and positive definite approximation to the Hessian matrix and  satisfies the standard secant equation
		$  B_k s_{k-1} = y_{k-1}. $ 	Since $d_k \in V_k$ can be expressed as  \begin{equation}\label{eq:dkform}  d    = {{u}{g_k} + {v }{s_{k - 1}}}, \end{equation} where $ u,v \in \mathbb  R $, by  substituting  (\ref{eq:dkform}) into (\ref{eq:QuaModel}) and using the standard secant equation, we     arrange      (\ref{eq:QuaModel})  as the following form :
		\begin{equation}\label{eq:NewQuaModel}\mathop {\min }\limits_{u,v \in \mathbb R} \;  {\left( {\begin{array}{*{20}{c}}
						{{{\left\| {{g_k}} \right\|}^2}}\\
						{g_k^T{s_{k - 1}}}
					\end{array}} \right)^T}\left( {\begin{array}{*{20}{c}}
					u\\
					v
				\end{array}} \right) + \frac{1}{2}{\left( {\begin{array}{*{20}{c}}
					u\\
					v
				\end{array}} \right)^T}\left( {\begin{array}{*{20}{c}}
				{{\rho _k}}&{g_k^T{y_{k - 1}}}\\
				{g_k^T{y_{k - 1}}}&{s_{k - 1}^T{y_{k - 1}}}
			\end{array}} \right)\left( {\begin{array}{*{20}{c}}
			u\\
			v
		\end{array}} \right),	\end{equation}
	where $ \rho_k \approx g_k^TB_kg_k $, namely, $ \rho_k $ is the estimate of $ g_k^TB_kg_k $.

	At first  the SMCG methods  received    little attention. For example, Andrei \cite{Andrei2014An}
	presented an efficient  SMCG method, where the search direction is generated over $ -g_k+ Span\left\lbrace  s_{k-1} , y_{k-1} \right\rbrace $;
	based on \cite{Andrei2014An}, Yang et al. \cite{Yang2017A}
	developed another SMCG method, in which the search direction is generated  over $ -g_k+ Span\left\lbrace  s_{k-1} , s_{k-2} \right\rbrace $.
	A   significant work about the SMCG method was given by Dai and Kou \cite{Dai2016A} in 2016. More specially,   Dai and Kou  established the finite termination for two dimensional convex quadratic functions of the SMCG method and presented a  Barzilai-Borwein conjugate gradient (BBCG) methods  with an efficient   estimate of the parameter $ \rho_k $:
	\begin{equation}\label{eq:BBCG3rhok} {\rho^{BBCG3} _k} = \frac{3}{2}\frac{{{{\left\| {{y_{k - 1}}} \right\|}^2}}}{{s_{k - 1}^T{y_{k - 1}}}}{\left\| {{g_k}} \right\|^2}  \end{equation}
	based on the BB method \cite{Dai2016A}.
	Motivated by  the SMCG method and $ {\rho^{BBCG3} _k} $, Liu and Liu \cite{Liu2019An} extended  the BBCG3 method  to general unconstrained optimization and presented an efficient subspace minimization conjugate gradient  method (SMCG$ \_ $BB) with the generalized   Wolfe line search. Since then, a lot of   SMCG methods   emerged for unconstrained optimization.   Based on \cite{Liu2019An}, Li  et al.  \cite{Li2019A} presented a new SMCG method based on  conic model and quadratic model; Wang  et al. \cite{Wang2020A} proposed a   new SMCG method based on  tensor model and quadratic model;    Zhao et al. \cite{Zhao2020New}  presented a  new SMCG method based on  regularization model     and quadratic model, and the numerical results in \cite{Zhao2020New} indicated these  SMCG methods is very efficient. 
	Recently, Sun et al. \cite{Sun2021A}    proposed  some accelerated   SMCG methods based on \cite{Zhao2020New}.  More advance about subspace minimization conjugate gradient   method can be referred  \cite{Li2018A,Zhang2019A}.


	Subspace minimization conjugate methods  are a class of efficient iterative methods for unconstrained optimization. On the one hand, the search direction of  SMCG method   is often parallel to the HS conjugate gradient method \cite{Dai2016A}.  On the other hand, traditional conjugate gradient method with  $ d_k =-g_k+ \beta_k d_{k-1} $ is only the special case of  SMCG method with $ d_k =u_k g_k+ v_k  s_{k-1}$. In other words, SMCG methods  can not only      inherit some important properties of traditional conjugate gradient methods but also have more choices for scaling the gradient $ g_k $ by $ u_k $, which will induce  that  SMCG method without the exact line search can enjoy  some additional   nice theoretical properties such as the finite termination for two dimensional convex functions   due to the term $ u_k  $ in the search direction compared to the traditional conjugate gradient methods.  
	In addition, SMCG methods   have also illustrated   nice numerical performance \cite{Liu2019An,Zhao2020New}. Based on the observation,  SMCG methods have  great potentiality and should be received more attention.
	
However,	
		the   estimate $ \rho_k $  of $   g_k^TB_kg_k $    must be      determined   before calculating the search direction. The parameter     $ \rho_k $ is very  crucial to the property and  the numerical performance of   SMCG methods, and   we still  do not understand how the parameter     $ \rho_k $ affects the numerical behavior of the SMCG method.    It is thus a great   challenge to determine the parameter properly.

 A simple analysis for the choice of $ \rho_k $ is given here. In  \eqref{eq:NewQuaModel}, the term $ g_k^T{y_{k-1}} $ and $ s_{k-1}^T{y_{k-1}} $     are obtained by using  the standard secant equation to eliminate $B_k$, namely, $g_k^T{y_{k-1}} = g_k^TB_k{s_{k-1}}$ and $  {s_{k-1}}^T{y_{k-1}} = {s_{k-1}^T}B_k{s_{k-1}}$.
For $ g_k^TB_kg_k $, we can not use the standard secant equation to eliminate $B_k$, which implies that $B_k$  must be given or estimated before computing the search direction.   Is the matrix $ B_k $ of $ g_k^TB_kg_k $   required to satisfy  the standard secant equation  ? If yes, some memoryless   quasi-Newton updating formulae can be applied to generate $ B_k $. It is however observed that the resulting $ \rho_k $ can not bring the desired numerical effect \cite{Dai2016A}. If not,  the matrices $ B_k $ in $  g_k^TB_kg_k $ and in $ g_k^T{y_{k-1}} $ are inconsistent, and we   do not know what will happen   even if the resulting estimate $ \rho_k $ is efficient. For example, in the   efficient choice  $ {\rho^{BBCG3} _k} $, the $B_k$  estimated by $ \frac{3}{2}\frac{{{{\left\| {{y_{k - 1}}} \right\|}^2}}}{{s_{k - 1}^T{y_{k - 1}}}}I $   is not satisfied  the standard secant equation, while the  $B_k$ in $g_k^T{y_{k-1}}$ satifies    the standard secant equation. And we do not konw why  $ {\rho^{BBCG3} _k} $ is so efficient in this way.

 It induces a challenge in determining the parameter $ \rho_k $ properly, which causes that it  is far from consensus for   good choice of the parameter $ \rho_k $ so far.	 As a result, it is no doubt that the choice of the parameter $ \rho_k $ is  a obstacle  for   the development of SMCG methods. A question is naturally to be asked:  can we develop an efficient  SMCG method without determining the parameter $ \rho_k $ ?
	
In the paper we do not focus on exploiting new choices for $ \rho_k $ since it is   difficult to   determine  it properly, as mentioned above. 	Instead, 	we are interested to focus on the    above question and study a new subspace minimization conjugate method without determining the important parameter $ \rho_k $.  Motivated by Dai-Kou conjugate gradient method \cite{Dai2016A}, we project the search direction of the memoryless quasi-Newton method into the subspace spanned by the current gradient and the latest search direction  and develop   a new SMCG method for unconstrained optimization. The new search direction is proved to be  descent.  It is remarkable that the SMCG method without any line  search   enjoys finite termination for two dimensional convex quadratic functions.  With  the improved Wolfe line search,  the convergence  of the proposed method for general nonlinear functions is established under the standard assumptions. Numerical experiments on the 147 test functions from the CUTEst library \cite{Gould2001CUTEr} indicates the proposed method  is very promising.

	The remainder of this  paper is organized  as follows.   We develop a new SMCG method for unconstrained optimization and exploit  some  important properties of the new search direction in Section 2.  In Section 3 we establish   the global convergence  of the proposed method for general nonlinear functions under the standard assumptions.   Some numerical experiments are conducted in Section 4.    Conclusions    are given in the last section.

	\section{New    subspace minimization  conjugate gradient method independent of  the parameter $ \rho_k $}

	In the section,   we first derive the new search directions, analyze their important properties, develop an adaptive scaling factor and present a new subspace minimization  conjugate gradient method independent of  the parameter $ \rho_k $ for unconstrained optimization.
	
	
	\subsection{The proposed search directions and  their important properties}

 	We are interested to
	the self-scaling memoryless
	BFGS method by Perry \cite{Perry1977A} and Shanno \cite{Shanno1978On}, where the search direction $\bar d_k^{PS}   $ is given by  

	%
	%
	%
	\[\bar d_k^{PS} =  - \frac{1}{{{\tau _k}}}{g_k} + \left[ {\frac{{g_k^T{y_{k - 1}}}}{{{\tau _k}s_{k - 1}^T{y_{k - 1}}}} - \left( {1 + \frac{{{{\left\| {{y_{k - 1}}} \right\|}^2}}}{{{\tau _k}s_{k - 1}^T{y_{k - 1}}}}} \right)\frac{{g_k^T{s_{k - 1}}}}{{s_{k - 1}^T{y_{k - 1}}}}} \right]{s_{k - 1}} + \frac{1}{{{\tau _k}}}\frac{{g_k^T{s_{k - 1}}}}{{s_{k - 1}^T{y_{k - 1}}}}{y_{k - 1}}.\]
Here $ \tau_k>0 $ is the scaling parameter.	The scaling memoryless quasi-Newton method is indeed three-term conjugate gradient method.  Specially, if the line search is exact, namely, $ g_k^T{s_{k-1}} =0 $, then the search direction $\bar d_k^{PS} $ is   HS conjugate gradient direction with scaling factor $ \frac{1}{{{\tau _k}}} $. It is not difficult to see that the   search direction $\bar d_k^{PS} $ only satisfies the following Dai-Liao conjugate condition \cite{DaiLiao2001}, namely,
	\[ {\left( {\bar d_k^{PS}} \right)^T}{y_{k - 1}} =  - {t_k}g_k^T{s_{k - 1}}, \;\;\text{where} \;\; t_k=1.\]
	As we know, the adaptive choice for $ t_k $ in Dai-Liao conjugate gradient methods \cite{DaiLiao2001} is usually more efficient than the prefixed choice. In addition, some famous and efficient conjugate gradient methods such as HZ conjugate gradient method \cite{Hager2005A}  and Dai-Kou conjugate gradient method \cite{Dai2013A} are  Dai-Liao  conjugate gradient methods with   adaptive    parameters $ t_k $. Therefore, to impose the search direction $\bar d_k^{PS} $ to satisfy the  Dai-Liao conjugate condition with adaptive    parameter, we
	multiple   $ { \bar {d}_k^{PS}} $ by $  {{{\tau _k}}} $ and obtain the following direction:	
	\begin{equation}\label{eq:d(PS)}d_k^{PS} =  - {g_k} + \left[ {\frac{{{g_k^T}{y_{k - 1}}}}{{s_{k - 1}^T{y_{k - 1}}}} - \left( {{\tau _k} + \frac{{{{\left\| {{y_{k - 1}}} \right\|}^2}}}{{s_{k - 1}^T{y_{k - 1}}}}} \right)\frac{{g_k^T{s_{k - 1}}}}{{s_{k - 1}^T{y_{k - 1}}}}} \right]{s_{k - 1}} + \frac{{g_k^T{s_{k - 1}}}}{{s_{k - 1}^T{y_{k - 1}}}}{y_{k - 1}}.
	\end{equation}
	Obviously, the search direction $ d_k^{PS}  $	satisfies Dai-Liao conjugate condition  	$  \left( {d_k^{PS}}\right) ^T{y_{k - 1}} =  - \tau_k g_k^T{s_{k - 1}}.   $ Noted that the scaling factor $ \tau_k $ is indeed the adaptive parameter in Dai-Liao conjugate gradient method. The   self-scaling memoryless
	BFGS method by Perry \cite{Perry1977A} and Shanno \cite{Shanno1978On} has been  applied successfully   to generate the famous and   efficient  Dai-Kou conjugate gradient method \cite{Dai2013A}.

	The search direction in subspace minimization  conjugate gradient method is usually generated in the subspace $ Span \left\lbrace   g_k, s_{k-1} \right\rbrace$, which means that 	
	$ d_{k }=u_k g_{k } +v_k s_{k-1}  $, where $ u_k $ and $ v_k $ are undetermined parameters. Different from \eqref{eq:NewQuaModel}   requiring to  estimate the parameter $ \rho_k $,  based on the search direction $ d_k^{PS} $, we will give a new way to derive $ u_k $ and $ v_k $  without requiring to  estimate the parameter $ \rho_k $. .
	
	We   consider the case  that $ g_k  $ is not parallel to $ s_{k-1} $, namely,
	\begin{equation} \label{eq:wkparalel}
		{\overline \omega _{k}}{\rm{ = }}\frac{{{{\left( {g_k^T{s_{k - 1}}} \right)}^2}}}{{{{\left\| {{g_k}} \right\|}^2}{{\left\| {{s_{k - 1}}} \right\|}^2}}}\le  \xi_1,
	\end{equation}
	where $0< \xi_1 <1 $ is close to 1. Otherwise, the search direction is naturally set to be    $d_k=  -g_k $.


	By projecting the search direction $ d_k^{PS} $ into the subspace  $ Span \left\lbrace   g_k, s_{k-1} \right\rbrace$, we get the  following subproblem:
	\begin{equation} \label{eq:projectionminPro}
		\mathop {\;\;\;\;\;\;\min }\limits_{\;\;\;\;d_{k }=u_k g_{k } +v_k s_{k-1} }\;\left\| {d_k^{PS}} - {d_{k }} \right\|_2^2.
	\end{equation}	
	
	\noindent	Solving the   subproblem \eqref{eq:projectionminPro}  yields the search direction
	\begin{equation} \label{eq:direction0}
		d_k ={u_k}{g_k} + {v_k}{s_{k - 1}},
	\end{equation}
	where
	\begin{equation} \label{eq:uk1}
		{u_{k  }} =  - 1 + \frac{{g_{k  }^T{y_{k-1}}g_{k  }^T{s_{k-1}}}}{{s_{k-1}^T{y_{k-1}}{{\left\| {{g_{k  }}} \right\|}^2}}} - \frac{{{{\left\| {{g_{k }}} \right\|}^2}{{\left( {g_{k  }^T{s_{k-1}}} \right)}^2} - g_{k  }^T{y_{k-1}}{{\left( {g_{k  }^T{s_{k-1}}} \right)}^3}/s_{k-1}^T{y_{k-1}}}}{{{{\left\| {{g_{k  }}} \right\|}^2}\left[ {{{\left\| {{g_{k }}} \right\|}^2}{{\left\| {{s_{k-1}}} \right\|}^2} - {{\left( {g_{k  }^T{s_{k-1}}} \right)}^2}} \right]}},
	\end{equation}		
	\begin{equation}  \label{eq:vk1}
		{v_{k  }} = \frac{{g_{k  }^T{y_{k-1}}}}{{s_{k-1}^T{y_{k-1}}}} + \frac{{{{\left\| {{g_{k  }}} \right\|}^2}g_{k  }^T{s_{k-1}} - g_{k  }^T{y_{k-1}}{{\left( {g_{k }^T{s_{k-1}}} \right)}^2}/s_{k-1}^T{y_{k-1}}}}{{{{\left\| {{g_{k  }}} \right\|}^2}{{\left\| {{s_{k-1}}} \right\|}^2} - {{\left( {g_{k }^T{s_{k-1}}} \right)}^2}}} - \left( {{\tau _k} + \frac{{{{\left\| {{y_{k-1}}} \right\|}^2}}}{{s_{k-1}^T{y_{k-1}}}}} \right)\frac{{g_{k  }^T{s_{k-1}}}}{{s_{k-1}^T{y_{k-1}}}}.
	\end{equation}
	
\noindent	It is not difficult to see that $ u_k $ and $ v_k $ can be   rewritten as the following forms:	
		\begin{equation}\label{eq:newukvk}
		{u_k}   =  \frac{1}{{1 - {\overline \omega _k}}}\left( { - 1 + \frac{{g_k^T{y_{k - 1}}g_k^T{s_{k - 1}}}}{{s_{k - 1}^T{y_{k - 1}}{{\left\| {{g_k}} \right\|}^2}}}} \right),\;			v_k  =   \frac{{1 - 2\bar \omega _k }}{{1 - \bar \omega _k }}\frac{{g_k^T y_{k - 1} }}{{s_{k - 1}^T y_{k - 1} }} - \left( {\tau _k  + \frac{{\left\| {y_{k - 1} } \right\|^2 }}{{s_{k - 1}^T y_{k - 1} }}} -  \frac{{s_{k - 1}^T y_{k - 1} }}{(1 - \overline \omega _k ) {\left\| {s_{k - 1} } \right\|^2 }}  \right)\frac{{g_k^T s_{k - 1} }}{{s_{k - 1}^T y_{k - 1} }},
		\end{equation}
		which are similar to the forms of conjugate gradient method. The new search direction \eqref{eq:newukvk} can be regarded the extension of Dai-Kou conjugate gradient direction.
	
	
	%
	
 	It is noted that the parameter  $ \tau_k $  in \eqref{eq:vk1} is the scaling factor in the memoryless quasi-Newton method, which is crucial to the numerical  performance of the corresponding   methods. There are various choices for $  \tau_k $, and in the    analysis on  the descent property and global convergence,    the following choices 	
	\begin{equation}\label{eq:sometauk}
		{\tau _k^{B}} = \frac{{s_{k - 1}^T{y_{k - 1}}}}{{{{\left\| {{s_{k - 1}}} \right\|}^2}}},\;\;\;\;\;\;{\tau _k^{H}} = \frac{{{{\left\| {{y_{k - 1}}} \right\|}^2}}}{{s_{k - 1}^T{y_{k - 1}}}},\;\;\;\;\;\;{\tau _k^{(1)}} = 1 
	\end{equation}
	are considered.
We also give an adaptive choice of $ \tau_k $ in Section 3.2 based on Theorem \ref{thm:Thorem1}.
	
	\begin{rem} \label{Remark1}
		If the line search is exact, namely, $ g_k^Ts_{k-1} =0 $, then it follows that $ u_k =-1 $	and $ v_k=\frac{g_k^Ty_{k-1}}{ {s_{k-1}^T{y_{k-1}}}} $ or  $ v_k=\frac{g_k^Ty_{k-1}}{ \alpha_{k-1}\left\| {g_{k-1} }\right\| ^2} $,  which mean the  search direction \eqref{eq:direction0} reduces to the HS  or PRP conjugate gradient  direction.
	\end{rem}
	\begin{rem}  \label{Remark2}The search direction \eqref{eq:direction0} satisfies the Dai-Liao conjugate condition
		$$d_k^T{y_{k - 1}} = \left[ {\frac{{\frac{{{{\left( {g_k^T{y_{k - 1}}} \right)}^2}{{\left\| {{s_{k - 1}}} \right\|}^2}}}{{s_{k - 1}^T{y_{k - 1}}}} - 2g_k^T{y_{k - 1}}g_k^T{s_{k - 1}} + {{\left\| {{g_{k - 1}}} \right\|}^2}s_{k - 1}^T{y_{k - 1}}}}{{{\Delta _k}}} - \left( {{\tau _k} + \frac{{{{\left\| {{y_{k - 1}}} \right\|}^2}}}{{s_{k - 1}^T{y_{k - 1}}}}} \right)} \right]g_k^T{s_{k - 1}} \buildrel \Delta \over = {t_k}g_k^T{s_{k - 1}}.$$
	\end{rem}
	
	We will establish  an interesting property---the finite termination of the SMCG method with \eqref{eq:CGform} and \eqref{eq:direction0} in the following theorem.
	
	\begin{theorem} \label{thm:Thorem1}
		Consider the SMCG method    \eqref{eq:CGform} and \eqref{eq:direction0} with $ \tau_k =1 $ for the convex quadratic function $q\left( x \right) = \frac{1}{2}{x^{\mathop{\rm T}\nolimits} }Ax + {b^{\mathop{\rm T}\nolimits} }x,{\rm{   }}x \in {\mathbb{R}^2}$,   where $A \in {\mathbb{R}^{2 \times 2}}$ is a symmetric and positive definite matrix and $b \in {\mathbb{R}^2}$. Assume that    $ d_0=-\alpha_0g_0 $, where $\alpha_0$ is the exact stepsize. Then, we must have that ${g_j} = 0$ for some $j \le 3$.
	\end{theorem}
	
	\begin{proof}
		
		~ Assume that ${g_j} \ne 0$ for $j = 0,1,2$. Since the first step is a Cauchy steepest descent step, we   know
		\begin{equation}\label{eq:g_2^T{s_1}=0}
			g_1^T{s_0} = 0.
		\end{equation}
		
		\noindent By \eqref{eq:direction0},  $ s_1=d_1 $ and Remark \ref{Remark2}, we have that
		\begin{align*} s_1^T{y_0}    = d_1^T{y_0}
			=  t_1 g_1^T{s_0}   = 0,
		\end{align*}
		where $ t_1 $ is given by Remark \ref{Remark2}. Thus,
		\begin{equation}\label{eq:s_2^T{y_1}=0}
			y_1^T{s_0} = s_1^TA{s_0} = 0. 	
		\end{equation}
		Since $n = 2,{s_0} \ne 0 \; \text{and} \;{y_1} = {g_2} - {g_1},$ we know from \eqref{eq:g_2^T{s_1}=0} and \eqref{eq:s_2^T{y_1}=0}    that ${g_1},{g_2}$  ${y_1}$ are collinear and there must exist some real number $a \ne 0$ such that
		\begin{equation}\label{eq:{y_1} = a{g_2}.}
			{y_1} = a{g_2}.
		\end{equation}
		
		\noindent By  \eqref{eq:uk1} and \eqref{eq:vk1}, we have
		\begin{align*}
			{u_2} &=  - 1 + \frac{{g_2^T{y_1}g_2^T{s_1}}}{{{{\left\| {{g_2}} \right\|}^2}s_1^T{y_1}}} - \frac{{{{\left\| {{g_2}} \right\|}^2}g_2^T{s_1} - g_2^T{y_1}{{\left( {g_2^T{s_1}} \right)}^2}/s_1^T{y_1}}}{{{{\left\| {{g_2}} \right\|}^2}{{\left\| {{s_1}} \right\|}^2} - {{\left( {g_2^T{s_1}} \right)}^2}}}\frac{{g_2^T{s_1}}}{{{{\left\| {{g_2}} \right\|}^2}}}\\
			&=  - 1 + \frac{{a{{\left\| {{g_2}} \right\|}^2}g_2^T{s_1}}}{{a{{\left\| {{g_2}} \right\|}^2}g_2^T{s_1}}} - \frac{{{{\left\| {{g_2}} \right\|}^2}g_2^T{s_1} - a{{\left\| {{g_2}} \right\|}^2}{{\left( {g_2^T{s_1}} \right)}^2}/(ag_2^Ts_1 )}}{{{{\left\| {{g_2}} \right\|}^2}{{\left\| {{s_1}} \right\|}^2} - {{\left( {g_2^T{s_1}} \right)}^2}}}\frac{{g_2^T{s_1}}}{{{{\left\| {{g_2}} \right\|}^2}}}\\
			&=- 1 + 1 + 0\\
			& =  0,\\
			{v_2} & = \frac{{g_2^T{y_1}}}{{s_1^T{y_1}}} - \left( {{\tau _k} + \frac{{{{\left\| {{y_1}} \right\|}^2}}}{{s_1^T{y_1}}}} \right)\frac{{g_2^T{s_1}}}{{s_1^T{y_1}}} + \frac{{{{\left\| {{g_2}} \right\|}^2}g_2^T{s_1} - g_2^T{y_1}{{\left( {g_2^T{s_1}} \right)}^2}/s_1^T{y_1}}}{{{{\left\| {{g_2}} \right\|}^2}{{\left\| {{s_1}} \right\|}^2} - {{\left( {g_2^T{s_1}} \right)}^2}}}\\
			& = \frac{{a{{\left\| {{g_2}} \right\|}^2}}}{{ag_2^T{s_1}}} - \left( {{\tau _k} + \frac{{{a^2}{{\left\| {{g_1}} \right\|}^2}}}{{ag_2^T{s_1}}}} \right)\frac{{g_2^T{s_1}}}{{ag_2^T{s_1}}} + 0\\
			& =  - \frac{{{\tau _k}}}{a},
		\end{align*}
		which implies that  ${s_2} = d_2= - \frac{{{\tau _k}}}{a}{s_1} .$
		Therefore,
		\begin{equation}\label{eq:g_4}
			{g_3} = {g_2} + {y_2} = {g_2} + A{s_2} = {g_2} - \frac{{{\tau _k}}}{a}A{s_1} = {g_2} - \frac{{{\tau _k}}}{a}{y_1} = (1 - {\tau _k})g_2.
		\end{equation}
		Since  ${\tau _k} ={\tau _k^{(1)}}= 1,$ we have   ${g_3} = 0,$ which completes the proof.  \qed

	\end{proof}
	
	\begin{rem}
		It follows  that 	the SMCG method    \eqref{eq:CGform} and \eqref{eq:direction0} with $ \tau_k =1 $ without any line search except the first Cauchy steepest descent  iteration enjoys the finite termination for two dimensional  convex quadratic functions. It seems  that it is not possible to obtain the same conclusion for traditional conjugate gradient methods without the exact line search.
	\end{rem}
	
	Together with  \eqref{eq:wkparalel}, let us    consider the following search direction:
	\begin{equation} \label{eq:direction1}
		{d_k} = \left\{ {\begin{array}{*{20}{c}}
				{ - {g_k},\;\;\;\;\;\;\;\;\;\;\;\;\;\;\;\;\;\;\; {\rm{if}}\;k = 0\;{\rm{or}}\;{\overline \omega _k} > {\xi _1},}\;\;\;\\
				{{u_k}{g_k} + {v_k}{s_{k - 1}},\;\;{\rm{otherwise,}}\;\;\;\;\;\;\;\;\;\;\;\;\;\;\;\;\;\;\; }
			\end{array}} \right.
		\end{equation}
		where $ {\overline \omega _k}, \; {u_k} $ and $ v_k $  are given by \eqref{eq:wkparalel}, \eqref{eq:uk1} and \eqref{eq:vk1}, respectively. We first do the following assumption:

		\begin{assumption} \label{Assumption}
			(i)The objective function $f $ is continuously differentiable  on $ {\mathbb   R^n} $;
			(ii) The level set $  {\cal L} = \left\{ {x|f\left( x \right) \le f\left( {{x_0}} \right) + \sum\limits_{k \ge 0} {{{\overline \eta  }_k}} } \right\}  $  is bounded, where $ \sum\limits_{k \ge 0} {{{\overline \eta  }_k}}< +\infty $; (iii) The gradient $g $  is  Lipschitz continuous on $ \mathbb R^n $, namely, there exists  a constant  $L>0$ such that		
			\begin{equation}\label{eq:Lipschitz continuous}
				\parallel g(x) - g(y)\parallel  \le L\parallel x - y\parallel ,\;\;\forall x,y \in  \mathbb R^{n}.
			\end{equation}	
			
		\end{assumption}

		Denote
		\begin{equation}\label{eq:pkandgammak}
			{p_{k}} = \frac{{{{\left\| {{y_{k-1}}} \right\|}^2}{{\left\| {{s_{k-1}}} \right\|}^2}}}{{{{\left( {s_{k-1}^T{y_{k-1}}} \right)}^2}}}, \;\;\;
			{\gamma _{k }} = {\tau _k}\frac{{{{\left\| {{s_{k-1}}} \right\|}^2}}}{{s_{k-1}^T{y_{k-1}}}}.
		\end{equation}
		The following lemma  discusses the descent property of the   search direction \eqref{eq:direction1}.
		
		\begin{lemma} \label{LemmaSuffDesc} Assume that $ f  $ satisfies Assumption \ref{Assumption} (iii), and  consider the   subspace minimization conjugate gradient methods \eqref{eq:CGform} and \eqref{eq:direction1} with any one of $ \tau_k $ in \eqref{eq:sometauk}. If $ s_{k-1}^Ty_{k-1}>0 $, then     	
			\begin{align} \label{eq:Suffdier0}
				g_k^T{d_k} <   0.
			\end{align}
			Furthermore, if   $ f $ is  uniformly convex, namely, there   exists $\mu  > 0$ such that
			\begin{equation} \label{eq:strongconvex0}
				{\left( {g\left( x \right) - g\left( y \right)} \right)^T}\left( {x - y} \right) \ge \mu {\left\| {x - y} \right\|^2},\;\;\;\forall x,y \in \mathbb R^n,
			\end{equation}
			then there  must exists $  c>0 $ such that
			\begin{align} \label{eq:Suffdier2}
				g_k^T{d_k} < -  c \left\| g_k \right\|^2.
			\end{align}
			
		\end{lemma}
		
		\begin{proof}
			We prove  the conclusion by dividing it into the following two cases.
			
			(i) $ d_k = -g_k $.  We know easily that  \eqref{eq:Suffdier0} and \eqref{eq:Suffdier2} both hold.
			
			(ii)  $ d_k = u_kg_k +v_k s_{k-1} $, where $ u_k $ and $ v_k $ are given by \eqref{eq:uk1} and \eqref{eq:vk1}, respectively.  It is not difficult to get that
			
			\begin{equation}\label{eq:g_k^T{d_k}}
				g_k^T{d_k} =  - {\left\| {{g_k}} \right\|^2} + \frac{{2g_k^T{s_{k - 1}}g_k^T{y_{k - 1}}}}{{s_{k - 1}^T{y_{k - 1}}}} - \left( {{\tau _k} + \frac{{{{\left\| {{y_{k - 1}}} \right\|}^2}}}{{s_{k - 1}^T{y_{k - 1}}}}} \right)\frac{{{{(g_k^T{s_{k - 1}})}^2}}}{{s_{k - 1}^T{y_{k - 1}}}},
			\end{equation}	
			which, in this sense,  implies that the above search direction     can be treated as
			\begin{equation}\label{d_k=- {H_k}{g_k}}
				\begin{aligned}
					{d_k}
					&=  - {g_k} + \left[ {\frac{{2g_k^T{y_{k - 1}}}}{{s_{k - 1}^T{y_{k - 1}}}} - \left( {{\tau _k} + \frac{{{{\left\| {{y_{k - 1}}} \right\|}^2}}}{{s_{k - 1}^T{y_{k - 1}}}}} \right)\frac{{g_k^T{s_{k - 1}}}}{{s_{k - 1}^T{y_{k - 1}}}}} \right]{s_{k - 1}}\\
					&=  - \left( {I - \frac{{2{s_{k - 1}}y_{k - 1}^T - \left( {{\tau _k} + \frac{{{{\left\| {{y_{k - 1}}} \right\|}^2}}}{{s_{k - 1}^T{y_{k - 1}}}}} \right){s_{k - 1}}s_{k - 1}^T}}{{s_{k - 1}^T{y_{k - 1}}}}} \right){g_k}\\
					& \buildrel \Delta \over =   - {H_k}{g_k}.
				\end{aligned}
			\end{equation}
			
			\noindent   For the symmetric part of  ${H_k}$:
			\begin{equation}\label{eq:symmetrize}
				{\bar H_k} = \frac{{{H_k} + H_k^T}}{2} = I - \frac{{{s_{k - 1}}y_{k - 1}^T + {y_{k - 1}}s_{k - 1}^T}}{{s_{k - 1}^T{y_{k - 1}}}} + \bar t_k\frac{{{s_{k - 1}}s_{k - 1}^T}}{{s_{k - 1}^T{y_{k - 1}}}},
			\end{equation}
			where $\bar t_k = {\tau _k} + \frac{{{{\left\| {{y_{k - 1}}} \right\|}^2}}}{{s_{k - 1}^T{y_{k - 1}}}},$   it is not difficult to verify  that
			\begin{equation}\label{eq:gHg}
				g_k^T{d_k} =  - g_k^T{H_k}{g_k} =  - g_k^T\left( {\frac{{{H_k} + H_k^T}}{2} + \frac{{{H_k} - H_k^T}}{2}} \right){g_k} =  - g_k^T{\bar H_k}{g_k} + 0 =  - g_k^T{\bar H_k}{g_k}.
			\end{equation}
			Now, we only need to    analyze the smallest   eigenvalues  of ${\bar H_k}.$  Rewriting ${\bar H_k}$    as \begin{equation}\label{eq:barHk}
				{\bar H_k} = I - \frac{{\left( {{y_{k - 1}} - \bar t_k{s_{k - 1}}} \right)s_{k - 1}^T}}{{s_{k - 1}^T{y_{k - 1}}}} - \frac{{{s_{k - 1}}y_{k - 1}^T}}{{s_{k - 1}^T{y_{k - 1}}}},
			\end{equation}
			we know that
			\begin{equation}\label{eq:detbarhk}
				\det \left( {{{\bar H}_k}} \right) =  - \frac{{{{\left\| {{y_{k - 1}}} \right\|}^2}{{\left\| {{s_{k - 1}}} \right\|}^2}}}{{{{\left( {s_{k - 1}^T{y_{k - 1}}} \right)}^2}}} + \frac{{\bar t_k{{\left\| {{s_{k - 1}}} \right\|}^2}}}{{s_{k - 1}^T{y_{k - 1}}}} = {\tau _k}\frac{{{{\left\| {{s_{k - 1}}} \right\|}^2}}}{{s_{k - 1}^T{y_{k - 1}}}},
			\end{equation}
			which implies that
			\begin{equation}\label{eq:eigenvalue product}
				{\lambda _{\min }}{\lambda _{\max }} = {\tau _k}\frac{{{{\left\| {{s_{k - 1}}} \right\|}^2}}}{{s_{k - 1}^T{y_{k - 1}}}}.
			\end{equation}
			It follows   from $
			trace\left( {{{\bar H}_k}} \right) = n - 2 + \bar t_k\frac{{{{\left\| {{s_{k - 1}}} \right\|}^2}}}{{s_{k - 1}^T{y_{k - 1}}}} = \left( {n - 2} \right) + {\lambda _{\min }} + {\lambda _{\max }}
			$   that
			\begin{equation}\label{eq:the sum of the eigenvalues}
				{\lambda _{\min }} + {\lambda _{\max }} = \frac{{{{\left\| {{y_{k - 1}}} \right\|}^2}{{\left\| {{s_{k - 1}}} \right\|}^2}}}{{{{\left( {s_{k - 1}^T{y_{k - 1}}} \right)}^2}}} + {\tau _k}\frac{{{{\left\| {{s_{k - 1}}} \right\|}^2}}}{{s_{k - 1}^T{y_{k - 1}}}}.
			\end{equation}
			Combining $ {s_{k - 1}^T{y_{k - 1}}}>0,$ \eqref{eq:eigenvalue product} and \eqref{eq:the sum of the eigenvalues} yields
			
			\begin{equation}  \label{eq:lambdamin}
				\begin{aligned}
					{\lambda _{\min }}& = \frac{{\frac{{{{\left\| {{y_{k - 1}}} \right\|}^2}{{\left\| {{s_{k - 1}}} \right\|}^2}}}{{{{\left( {s_{k - 1}^T{y_{k - 1}}} \right)}^2}}} + {\tau _k}\frac{{{{\left\| {{s_{k - 1}}} \right\|}^2}}}{{s_{k - 1}^T{y_{k - 1}}}} - \sqrt {{{\left( {\frac{{{{\left\| {{y_{k - 1}}} \right\|}^2}{{\left\| {{s_{k - 1}}} \right\|}^2}}}{{{{\left( {s_{k - 1}^T{y_{k - 1}}} \right)}^2}}} + {\tau _k}\frac{{{{\left\| {{s_{k - 1}}} \right\|}^2}}}{{s_{k - 1}^T{y_{k - 1}}}}} \right)}^2} - 4{\tau _k}\frac{{{{\left\| {{s_{k - 1}}} \right\|}^2}}}{{s_{k - 1}^T{y_{k - 1}}}}} }}{2}
					\\& = \frac{{{p_k} + {\gamma_k} - \sqrt {{{\left( {{p_k} + {\gamma_k}} \right)}^2} - 4{\gamma_k}} }}{2}
					\\& >0
					,
				\end{aligned}
			\end{equation}
			where $ p_k $ and $ \gamma_k $ are given by \eqref{eq:pkandgammak}.
			As a result, for any one of $ \tau_k $ in \eqref{eq:sometauk}, we have $ -g_k^Td_k = g_k^T \bar H_k g_k \ge  \lambda_{\min} {\left\| {{g_{k}}} \right\|}^2 >0 $, which implies that \eqref{eq:Suffdier0}.
			
			We next analyze the sufficient descent property of the search direction with different $ \tau_k $ in \eqref{eq:sometauk} when $f$ is uniformly convex.
			
			(a) $ \tau_k = \tau_k^{H}= \frac{{{{\left\| {{y_{k - 1}}} \right\|}^2}}}{{s_{k - 1}^T{y_{k - 1}}}}
			$.   We have that  $  {\gamma_k} =  {\tau _k}\frac{{{{\left\| {{s_{k-1}}} \right\|}^2}}}{{s_{k-1}^T{y_{k-1}}}} = p_k $ and ${\lambda _{\min }} = {p_k} - \sqrt {p_k^2 - {p_k}}.$
			Since
			\[\frac{{d{\lambda _{\min }}}}{{d{p_k}}} = 1 - \frac{{2p_k - 1}}{{2\sqrt {{p_k^2} - p_k} }} < 0,\;\;  \forall p_k >1,\]
			$ \lambda_{\min} $ is monotonically  decreasing in $  \left[ 1, +\infty \right)  $ and thus  $$ \lambda_{\min} > 1/2.$$
			Noted that when  $ \tau_k = \tau_k^{H}= \frac{{{{\left\| {{y_{k - 1}}} \right\|}^2}}}{{s_{k - 1}^T{y_{k - 1}}}}
			$, the sufficient descent property of the search direction is proved without the uniformly convexity condition \eqref{eq:strongconvex0}.

			(b) $ \tau_k = \tau_k^{B}=\frac{{s_{k - 1}^T{y_{k - 1}}}}{{{{\left\| {{s_{k - 1}}} \right\|}^2}}}$. We have that $  {\gamma_k} = \tau_k \frac{{{{\left\| {{s_{k-1}}} \right\|}^2}}}{{s_{k-1}^T{y_{k-1}}}}   = 1 $  and ${\lambda _{\min }} = \frac{{{p_k} + 1 - \sqrt {{{\left( {{p_k} + 1} \right)}^2} - 4} }}{2}.$
			Since
			\begin{equation} \label{eq:lambdamintauB}\frac{{d{\lambda _{\min }}}}{{d{p_k}}} = \frac{1}{2} - \frac{{p_k + 1}}{{2\sqrt {{{\left( {p_k + 1} \right)}^2} - 4} }} < 0,\;\;  \forall p_k >1,\end{equation}
			$ {\lambda _{\min }} $ is monotonically decreasing in $  \left[ 1, +\infty \right)  $. By Assumption \ref{Assumption} (iii) and \eqref{eq:strongconvex0}, we know that
			\begin{equation} \label{eq:pkbound0}
				{{p_k}}=   \left( \frac{{\left\| {{s_{k-1}}} \right\|\left\| {{y_{k-1}}} \right\|}}{{s_{k-1}^T{y_{k-1}}}} \right) ^2\le  \left(\frac{{L{{\left\| {{s_{k-1}}} \right\|}^2}}}{{s_{k-1}^T{y_{k-1}}}}  \right) ^2\le \frac{L^2}{\mu^2 }  .
			\end{equation}
			Therefore,
			\[{\lambda _{\min }} \ge \frac{{{L^2}/{\mu ^2} + 1 - \sqrt {{{\left( {{L^2}/{\mu ^2} + 1} \right)}^2} - 4} }}{2}.\]
			
			(c)  $ \tau_k = \tau_k^{(1)}=1$. We have $  {\gamma_k} = \tau_k \frac{{{{\left\| {{s_{k-1}}} \right\|}^2}}}{{s_{k-1}^T{y_{k-1}}}}   = \frac{{{{\left\| {{s_{k-1}}} \right\|}^2}}}{{s_{k-1}^T{y_{k-1}}}}   $ and ${\lambda _{\min }} = \frac{{{p_k} + {\gamma _k} - \sqrt {{{\left( {{p_k} + {\gamma _k}} \right)}^2} - 4{\gamma _k}} }}{2}$.  Since
			\[\frac{{\partial {\lambda _{\min }}}}{{\partial {p_k}}} = \frac{1}{2} - \frac{{{p_k} + {\gamma _k}}}{{2\sqrt {{{\left( {{p_k} + {\gamma _k}} \right)}^2} - 4{\gamma _k}} }},\]
			$ {\lambda _{\min }} $ is monotonically decreasing with respect to $p_k$.
			
			It follows from \eqref{eq:pkbound0}  that $p_k  \le L^2 \gamma _k^2 $. Let  ${\bar L}$ be any    value such that  $\left( {\gamma _k  + L^2 \gamma _k^2 } \right)^2  - 4\gamma _k  \ge 0$.  Thus,   $p_k  \le \max \left\{ {L^2 ,\bar L^2 } \right\}\gamma _k^2  \buildrel \Delta \over = \tilde L^2 \gamma _k^2 $. As a result,
			\begin{equation}\label{eq:lambdamintauq}
				\lambda _{\min }   \ge \frac{{\gamma _k  + \tilde L^2 \gamma _k^2  - \sqrt {\left( {\gamma _k  + \tilde L^2 \gamma _k^2 } \right)^2  - 4\gamma _k } }}{2} \buildrel \Delta \over = \bar \phi \left( {\gamma _k } \right) .
			\end{equation}
			It follows from  \eqref{eq:strongconvex0}   that $ \gamma _k  \le \frac{1}{\mu}  $.
			It is not difficult to verify that $\bar \phi '\left( {\gamma _k } \right) < 0$, which implies that    $\bar \phi \left( {\gamma _k } \right)$ is monotonically decreasing in $ \left(0,\frac{1}{\mu} \right]  $  and
			\[{\lambda _{\min }} \ge \bar \phi \left( {\frac{1}{\mu }} \right) = \frac{{1/\mu  + {{\tilde L}^2}/{\mu ^2} - \sqrt {{{\left( {{1/\mu } + {{\tilde L}^2}{/\mu ^2}} \right)}^2} - 4/\mu } }}{2}.\]
			%
			
			
			In conclusion, for any one of $ \tau_k $ in \eqref{eq:sometauk}, there must exists $   c>0 $ such that $ {\lambda _{\min }} \ge   c $, which implies that   $$  g_k^Td_k = -g_k^T \bar H_k g_k \le  -\lambda_{\min} {\left\| {{g_{k}}} \right\|}^2 \le -  c {\left\| {{g_{k}}} \right\|}^2. $$
			It completes the proof.  \qed

		\end{proof}

		Powell \cite{Powellexample1984} constructed a counterexample showing that the PRP method with
		exact line search may not converge for general nonlinear functions. It follows from  Remark \ref{Remark1} that Powell's example can also be used to show that the
		method \eqref{eq:CGform} and \eqref{eq:direction1}  with any one of $ \tau_k $ in \eqref{eq:sometauk}  may not converge for general  nonlinear functions.  Therefore,   motivated the truncation form in  \cite{Dai2013A}, we     truncate similarly $ v_k $ in  \eqref{eq:vk1} as
		\begin{equation} \label{eq:Tru-vketak}{{\bar v}_k} = \max\left\lbrace v_k, \eta_k \right\rbrace, \;\; \end{equation}
	 where
	 \begin{equation}
	 \label{eq:etakk}  {\eta _k} = -l_k \frac{{\left| g_k^T{s_{k - 1}}\right| }}{{{{\left\| {{s_{k-1}}}\right\|}^2}}},\;\;\;   l_k  = \left\{ {\begin{array}{*{20}{c}}
	 	{\xi_2,\;\;\;\;\;\;\;\;\;\;\;\;\;\;\;\;\;\;\;\;\;\;\;\;\;\;\;\;\;\;\;\;\;\;\;\;\;\;\;\;\;\;\;\;\;\;\;\; \;\;\;\;\;\; \;\;\;\;\;\;\;\;\;\;\;\text{if}\; g_k^T{s_{k - 1}} \le 0,}\\
	 	{\min\left\{  \max \left\{ {{\bar{\bar{\xi_2}}}, - 1 + \left( {1 + {u_k}} \right)/\overline \omega_k  } \right\},\bar \xi_2\right\},\;\;\;\;\;\;\text{otherwise}.}  \;\;\;\;\;
	 	\end{array}} \right.   \end{equation}
	 Here $ 0<\bar \xi_2<1$, and $0< \xi_2<1.  $
			
			When applying the conjugate gradient methods with the exact line search to solve   quadratic minimization problems,     the sequence of the corresponding gradients is orthogonal, namely, $ g_k^T g_{j} = 0, \; 0 \le j \le k-1.$
			For general nonlinear   functions,   one also hope that  $\left|  g_k^T g_{k-1}   \right| $ may not be far from 0. When $\left|  g_k^T g_{k-1}   \right| > \xi \left\| g_k\right\|^2  $, where $0< \xi <1 $, Powell \cite{Powell1977Restart} suggested that   the search direction should be restarted with $ d_k = -g_k $. Powell's restart strategy is quite efficient   and has been used widely in the numerical implementation of conjugate gradient methods.   As a result, if the   condition
			\begin{equation} \label{eq:restart1}  -\eta_1 {\left\| {{g_k}} \right\|}^2 \le   g_k^T g_{k-1}   \le  \eta_2  {\left\| {{g_k}} \right\|}^2,\;\; \;0<\eta_2<1, \;\eta_1>\eta_2
				\end{equation}
			holds, our method will be also restarted with $ -g_k $. It   follows from  \eqref{eq:restart1}  that 	\begin{equation} \label{eq:gy>0}0<\left(1-\eta_2 \right)  \le \dfrac{g_k^T y_{k-1}}{\left\| g_k \right\|^2 } \le\left(1+\eta_1 \right)   .  \end{equation}.

			Therefore, the search direction is summarized below:
			\begin{equation}\label{eq:Truconditon2}{d_k} = \left\{ {\begin{array}{*{20}{c}}
						{{-g_k},\;\;\;\;\;\;\;\;\;\;\;\;\;\;\;\;\;\;{\rm{if}}\;k = 0\; {\rm{or}}\;{{\overline \omega }_k} > {\xi _1}\;{\rm{or}}}\; \eqref{eq:restart1}\; \text{holds},\;\;\;\;\; \;\;\;\\
						{{  u_k}{g_k} + \bar v_k  {s_{k - 1}},\;\; {\rm{otherwise,}}\;\;\;\;\;\;\;\;\;\;\;\;\;\;\;\;\;\;\;\;\;\;\;\; \;\;\;\;\;\;\;\;\;\;\;\;\;\;\;\;\;\;\;\;\;\;}
					\end{array}} \right.\end{equation}
				where
				$ \;{{\bar v}_k}  \;   {\rm{and}}\;{\overline \omega _k}  $ are given by \eqref{eq:Tru-vketak}  and \eqref{eq:wkparalel}, respectively. 

				\begin{lemma} \label{LemmaSuffDesc2} Assume that $ f  $ satisfies Assumption \ref{Assumption} (iii), and  consider the   subspace minimization conjugate gradient methods \eqref{eq:CGform} and \eqref{eq:Truconditon2} with any one of $ \tau_k $ in \eqref{eq:sometauk}.  If $ s_{k-1}^Ty_{k-1}>0 $, then     	    there   exists   $ c>0 $ such that   	
					\begin{align} \label{eq:Suffdier}
						g_k^T{d_k}\le -c {\left\| {{g_k}} \right\|^2}.
					\end{align}
					
				\end{lemma}
				
				\begin{proof} We prove it in the following cases.
					
					(i) $ d_k = -g_k $. We have $ g_k^T{d_k} =- \left\| {{g_{k  }}} \right\|^2 $.
					

					(ii) $ d_k = u_kg_k +\eta_ks_{k-1}.  $   If $ g_k^Ts_{k-1} \le 0 $, then by \eqref{eq:newukvk} and \eqref{eq:etakk} we have that
					\[\begin{array}{l}
					g_k^T{d_k} = \left( { - {{\left\| {{g_k}} \right\|}^2} + \frac{{g_k^T{y_{k - 1}}g_k^T{s_{k - 1}}}}{{s_{k - 1}^T{y_{k - 1}}}}} \right)/\left( {1 - {{\overline \omega  }_k}} \right) + {l_k }{\overline \omega  _k}{\left\| {{g_k}} \right\|^2}\\
					\;\;\;\;\;\;\;\; =  - \left( {\dfrac{1}{{1 - {{\overline \omega  }_k}}} - {l_k  }{{\overline \omega  }_k} - \dfrac{{g_k^T{y_{k - 1}}g_k^T{s_{k - 1}}}}{{s_{k - 1}^T{y_{k - 1}}{{\left\| {{g_k}} \right\|}^2}}}} \right){\left\| {{g_k}} \right\|^2}.
					\end{array}\]		
			 It follows from \eqref{eq:wkparalel}, $ g_k^Ts_{k-1} \le 0 $ and    \eqref{eq:gy>0} that  $$  \frac{1}{{1 - {{\overline \omega  }_k}}} - {l_k }{\overline \omega  _k} - \frac{{g_k^T{y_{k - 1}}g_k^T{s_{k - 1}}}}{{s_{k - 1}^T{y_{k - 1}}{{\left\| {{g_k}} \right\|}^2}}} \ge \frac{1}{{1 - {{\overline \omega  }_k}}} -    {l_k }{\overline \omega  _k} \ge 1 -\xi_1.$$ Therefore, when if $g_k^Ts_{k-1}\ge 0$, we obtain that
						$$ g_k^T{d_k} \le -(1 -\xi_1) \left\| {{g_k}} \right\|^2. $$
			 	If $ g_k^Ts_{k-1} > 0 $, then
				\[\begin{array}{l}
				g_k^T{d_k} = {u_k}{\left\| {{g_k}} \right\|^2} - {l_k  }{\overline \omega  _k}{\left\| {{g_k}} \right\|^2}\\
				\;\;\;\;\;\;\;\; =  - \left( { - {u_k} + l_k {{\overline \omega  }_k}} \right){\left\| {{g_k}} \right\|^2}.
				\end{array}\]
	According to \eqref{eq:etakk}, we can easily have  that $ - {u_k} + l_k {\overline \omega  _k} \ge 1 - {\overline \omega  _k} \ge 1 - {\xi _1} $. Therefore, when $g_k^Ts_{k-1}> 0$, we have that
		$$ g_k^T{d_k} \le -\left(  1 - {\xi _1} \right)  \left\| {{g_k}} \right\|^2. $$
		 			
					
					(iii)  $ d_k =   u_k g_k +   v_k s_{k - 1} $. According to \eqref{eq:newukvk} and  \eqref{eq:Truconditon2}, we obtain that
					\begin{equation}\label{eq:v_k and eta_k}
						\begin{aligned}
							v_k
							&= \frac{{1 - 2\overline \omega _k }}{{1 - \overline \omega _k }}\frac{{g_k^T y_{k - 1} }}{{s_{k - 1}^T y_{k - 1} }} - \left( {\tau _k  + \frac{{\left\| {y_{k - 1} } \right\|^2 }}{{s_{k - 1}^T y_{k - 1} }}} \right)\frac{{g_k^T s_{k - 1} }}{{s_{k - 1}^T y_{k - 1} }} + \frac{1}{{1 - \overline \omega _k }}\frac{{g_k^T s_{k - 1} }}{{\left\| {s_{k - 1} } \right\|^2 }} \\
							&\ge -l_k \frac{{ \left| g_k^T s_{k - 1} \right|  }}{{\left\| {s_{k - 1} } \right\|^2 }}.
						\end{aligned}
					\end{equation}
					
					 If $g_k^T s_{k - 1}  \le 0$, then from \eqref{eq:g_k^T{d_k}}  and \eqref{eq:gy>0} we know that   $g_k^T d_k  \le  - \left\| {g_k } \right\|^2 $. So we only need to consider the case of  $g_k^T s_{k - 1}  > 0$. Multiplying both sides of  \eqref{eq:v_k and eta_k} by $\frac{{g_k^T s_{k - 1} }}{{\left\| {g_k } \right\|^2 }}$ yields
					$$
					\frac{{1 - 2\overline \omega _k }}{{1 - \overline \omega _k }}\frac{{g_k^T y_{k - 1} }}{{s_{k - 1}^T y_{k - 1} }}\frac{{g_k^T s_{k - 1} }}{{\left\| {g_k } \right\|^2 }} - \left( {\tau _k \frac{{\left\| {s_{k - 1} } \right\|^2 }}{{s_{k - 1}^T y_{k - 1} }} + \frac{{\left\| {y_{k - 1} } \right\|^2 \left\| {s_{k - 1} } \right\|^2 }}{{\left( {s_{k - 1}^T y_{k - 1} } \right)^2 }}} \right)\frac{{\left( {g_k^T s_{k - 1} } \right)^2 }}{{\left\| {g_k } \right\|^2 \left\| {s_{k - 1} } \right\|^2 }}+  \frac{\overline \omega _k}{{1 - \overline \omega _k }}   \ge -l_k \frac{{\left( {g_k^T s_{k - 1} } \right)^2 }}{{\left\| {g_k } \right\|^2 \left\| {s_{k - 1} } \right\|^2 }}.
					$$
					According to \eqref{eq:pkandgammak}, we have that
					\begin{align}\label{eq:**}
						\frac{{\frac{1}{{\overline \omega _k }} - 2}}{{\frac{1}{{\overline \omega _k }} - 1}}\frac{{g_k^T y_{k - 1} }}{{\left\| {g_k } \right\|^2 }}\frac{{g_k^T s_{k - 1} }}{{s_{k - 1}^T y_{k - 1} }} \ge \gamma _k  + p_k  - \frac{1}{{1 - \overline \omega _k }}  - l_k.
					\end{align}
					It follows from \eqref{eq:Suffdier0} in Lemma \ref{LemmaSuffDesc} and $ {g_k^T s_{k - 1} }>0  $ that $0 < \frac{{g_k^T s_{k - 1} }}{{s_{k - 1}^T y_{k - 1} }} < 1$. It follows from \eqref{eq:wkparalel} and \eqref{eq:**} that
					\begin{align}\label{eq:***}
						p_k  \le \gamma _k  + p_k  \le 1+ \eta_1 + l_k  + \frac{1}{{1 - \overline \omega _k }} \le  1+\eta_1 +\bar \xi_2 + \frac{1}{{1 - \xi _1 }} \buildrel \Delta \over = \xi _0. 	
					\end{align}
					
					We next derive the   conclusion for any one of  $ \tau_k $ in \eqref{eq:sometauk} based on  \eqref{eq:lambdamin} as follows.
					
					(a) $\tau _k  = \tau _k^H  =\frac{{\left\| {y_{k - 1} } \right\|^2 }}{{s_{k - 1}^T y_{k - 1} }}$. We have that   $ g_k^T{d_k} \le -0.5 \left\| {{g_{k  }}} \right\|^2 $ by Lemma \ref{LemmaSuffDesc}.

					(b)  $\tau _k  = \frac{{s_{k - 1}^T y_{k - 1} }}{{\left\| {s_{k - 1} } \right\|^2 }}$. According to Lemma \ref{LemmaSuffDesc}, we know that  $\lambda _{\min }   = \frac{{1 + p_k  - \sqrt {\left( {p_k  - 1} \right)\left( {p_k  + 3} \right)} }}{2}$ and $ {\lambda _{\min }} $ is monotonically decreasing in $  \left[ 1, +\infty \right)  $. Combining with \eqref{eq:***}, we obtain
					$$
					\lambda _{\min }  \ge  \frac{{1 + \xi _0  - \sqrt {\left( {\xi _0  - 1} \right)\left( {\xi _0  + 3} \right)} }}{2}      > 0.
					$$
					
					(c) $\tau _k  = 1$. According to Lemma \ref{LemmaSuffDesc}, we know that $$\lambda _{\min }   = \frac{{\gamma _k  + p_k  - \sqrt {\left( {\gamma _k  + p_k } \right)^2  - 4\gamma _k } }}{2}.$$
					Since $\frac{{\partial {\lambda _{\min }}}}{{\partial {p_k}}} =   {1 - \frac{{\gamma _k  + p_k }}{{\sqrt {\left( {\gamma _k  + p_k } \right)^2  - 4\gamma _k } }}}   < 0$,   $\lambda _{\min } $ is monotonically decreasing with respect to $p_k$ in $  \left[ 1, +\infty \right)  $. Similarly to Lemma \ref{LemmaSuffDesc}, we can obtain
					$$
					\lambda _{\min }   \ge \frac{{\gamma _k  + \tilde L^2 \gamma _k^2  - \sqrt {\left( {\gamma _k  + \tilde L^2 \gamma _k^2 } \right)^2  - 4\gamma _k } }}{2} \buildrel \Delta \over = \bar \phi \left( {\gamma _k } \right) ,
					$$
					where $ \tilde L $ is the same as that in\eqref{eq:lambdamintauq}.
					It is not difficult to verify that $\bar \phi '\left( {\gamma _k } \right) < 0$, which, together with $ \gamma_k \le \xi_0 $ implied by \eqref{eq:***},  yields that    $\bar \phi \left( {\gamma _k } \right)$ is monotonically decreasing and $$\lambda _{\min }  \ge \bar \phi \left( {\xi _0 } \right)= \frac{{{\xi _0} + {{\tilde L}^2}\gamma _k^2 - \sqrt {{{\left( {{\xi _0} + {{\tilde L}^2}\xi _0^2} \right)}^2} - 4{\xi _0}} }}{2} > 0 .$$
					
					In conclusion, for any one of $ \tau_k $ in \eqref{eq:sometauk}, there must exists $ c>0 $ such that $ {\lambda _{\min }} \ge   c $, which together with \eqref{eq:gHg} implies that
					$$  g_k^Td_k = -g_k^T \bar H_k g_k \le  -\lambda_{\min} {\left\| {{g_{k-1}}} \right\|}^2 \le -  c {\left\| {{g_{k-1}}} \right\|}^2.$$
					It completes the proof. \qed			
				\end{proof}	
				
				\subsection{Adaptive  choice of $ \tau_k $}
				
				The choice of $ \tau_k  $ is also  crucial to the search direction  \eqref{eq:Truconditon2}.
			   $ \tau_k $ is given by the following observation. From Theorem \ref{thm:Thorem1}, we know that the SMCG method \eqref{eq:CGform} and \eqref{eq:direction0}  with $ \tau_k   =1 $ and  without any line search  expect the first Cauchy steepest descent  iteration can enjoy the finite termination property when  the objective function $ f $ is 2 dimensional strictly convex quadratic function. Therefore, the choice of  $ \tau_k  =1 $ may be preferred in some cases.  
			   
			     According to \cite{Yuan1991A,Dai2002Modified},
				\begin{align} \label{eq:QuadraticUk}	{  \mu _{k}} = \left| {\frac{{2\left( {{   f_{k - 1}} - {  f_{k}} +   g_k^T{  s_{k - 1}}} \right)}}{{  s_{k - 1}^T{  y_{ k - 1}}}} - 1} \right|  \end{align}
				\noindent is  a   quantity measuring
				how $  f   $ is close to a quadratic on the line segment between $  x_{k-1}  $ and $  x_{k} $.
				If the following  condition \cite{Liu2017AnNA,Dai2002Modified} holds,   namely,   	
				\begin{align} \label{eq:judgequa2}{\mu _k} \le { \xi_3 } \;\;\;\;  \text{or}  \;\;\;\;  \max \left\{ {{\mu _k},{\mu _{k - 1}}} \right\} \le {\xi_4 },
				\end{align}
				where  $\xi_3 \; \text{and}\;\xi_4  $ are   small positives and $\xi_3 <\xi_4 $,   $   f  $ might be very close to a quadratic function on the line
				segment between $  x_{{k}-1}  $ and $  x_{k} $. Therefore, if \eqref{eq:judgequa2} and   the following condition
				\begin{align} \label{eq:TauCon} {{{\left\| {{g_k}} \right\|}^{\rm{2}}} \le {\rm{\xi_6}}\;{\rm{or}}\;\left( {{{\left\| {{g_k}} \right\|}^{\rm{2}}}{\rm{> \xi_6}}\;{\rm{and}}\;{{\left\| {{s_{k - 1}}} \right\|}^2} \le \xi_5 } \right)}
				\end{align}
				hold, then the search direction \eqref{eq:d(PS)}  should not be scaled, namely, $ \tau_k =1 $. Here, $ \xi_5,\;\xi_6,  >0.   $ It is  noted that the condition \eqref{eq:TauCon} means that the current iterative point $ x_k $ is close to the stationary point or close to the latest iterative point $ x_{k-1} $.
				
				Therefore,  $ \tau_k  $ is given by
				\begin{align} \label{eq:Trutauk}{\tau _k} = \left\{ {\begin{array}{*{20}{c}}
							{1,\;\;\;\;  {\rm{if}}\;\eqref{eq:judgequa2}\;\;{\rm{and}}\;  \eqref{eq:TauCon}\; \; {\rm{hold}},\;}\\
							{\tau _k^B,\;\; {\rm{otherwise}},}\;\;\;\;\;\;\;\;\;\;\;\;\;\;\;\;\; \;\;\;\;\;\;\; \;
						\end{array}} \right.\end{align}
	where $\tau _k^B$ is given by \eqref{eq:sometauk} 				
					
					\subsection{ The initial stepsize and the improved Wolfe line search  }
					
					It is universally accepted  that the choice of initial stepsize is of great importance
					for a line search method. Unlike general quasi-Newton methods, it is challenging to
					determine a suitable initial stepsize for a SMCG method. The initial stepsize in our method  is also similar to Algorithm 3.1 in \cite{Dai2013A}, and   the main difference lies in that   we replace
					\begin{equation} \label{eq:alpha0}
						\alpha _k^{\left( 0 \right)} = \max \left\{ {\varphi {\alpha _{k - 1}}, - 2\left| {{f_k} - {f_{k - 1}}} \right|/g_k^T{d_k}} \right\}
					\end{equation}
					in Step 1 of  Algorithm 3.1 in \cite{Dai2013A}   by
					\[\bar \alpha _k^{\left( 0 \right)} = \left\{ {\begin{array}{*{20}{c}}
						{ \alpha _k^{\left( 0 \right)},\;\;\;\;\;\;\;\;\;\;\;\;\;\;\;\text{if}\;{d_k} =  - {g_k}},\;\\
						{\min\left\{ {1,  \alpha _k^{\left( 0 \right)}} \right\},\;\text{if}\;{d_k} \ne  - {g_k}},
						\end{array}} \right.\]
					where $\alpha _k^{\left( 0 \right)} $ is given by \eqref{eq:alpha0}. The motivation behind is that the search direction \eqref{eq:direction0} is closest to the search direction  of the memoryless quasi-Newton method, which usually prefers the unit stepsize 1.
					
					The improved Wolfe line search proposed by Dai and Kou \cite{Dai2013A} is an quite efficient Wolfe line search, which can   avoid some
					numerical drawbacks of the original Wolfe line search. It aims to find the stepsize satisfying the following conditions:
					\begin{align}
						f\left( {{x_k} + {\alpha _k}{d_k}} \right)   \le f\left( {{x_k}} \right) + \min \left\{ {{\epsilon}\left| {f\left( {{x_k}} \right)} \right|,\delta \alpha_k g_k^T{d_k} + {\bar \eta _k}} \right\} ,
						\label{eq:Daiimprlinsear1} \\
						g{\left( {{x_k} + {\alpha _k}{d_k}} \right)^T}{d_k}  \ge \sigma g_k^T{d_k}.
						\label{eq:Daiimprlinsear2}
					\end{align}
					where $0 < \epsilon   $, $ 0 < \delta  < \sigma  < 1 $, $0 < \bar \eta_k  $ and $  \sum\limits_{k \ge 0} {{{\overline \eta  }_k}}  <  + \infty   $.	
					The above improved Wolfe line search   is   used in our   method.

					\subsection{The proposed method}
					
					Denote
					\begin{equation}\label{eq:restartPara}
						{r_{k - 1}} = \frac{{2\left( {{f_k} - {f_{k - 1}}} \right)}}{{{{\left( {{g_k} + {g_{k - 1}}} \right)}^T}{s_{k - 1}}}} - 1,\;\;\;{\overline r _{k - 1}} = {f_k} - {f_{k - 1}} - 0.5\left( {g_k^T{s_{k - 1}} + g_k^T{s_k}} \right).\end{equation}
					Similarly  to  the restart condition in \cite{Liu2019An,LiuDaiLiuCGOPT20},     if there are continuously many
					iterations such that  $ r_{k-1} $ or $\overline r_{k-1} $ is close to 0,  our  algorithm is also restarted  with $ -g_k$.
					
					The subspace minimization conjugate gradient method is described in  detail  as follow.
					
					\begin{algorithm}\label{Algorithm1}
						\centering
						\caption{Subspace Minimization Conjugate Gradient Methods (SMCG)}
						\begin{algorithmic}
							\STATE\textbf {Step 0.} Initialization. Given     ${x_0} \in \mathbb {R}^n,$   $\varepsilon  > 0,$  $\delta,$ ${\sigma },  $ $ \epsilon,\; {\epsilon_1}   $, $  {\xi _1},\; {\xi _2},\; {\xi _3},\; {\xi _4},\; {\xi _5},\; {\xi _6}, \text{MaxRestart},\; \text{MinQuad} $.
							Set   $ \text{IterQuad}=0,$  \\
							\quad	\qquad $\;\text{IterRestart}=0,\;k = 0, $\\
							\STATE\textbf {Step 1.} If $\left\| {{g_0}} \right\|_{\infty} \le \varepsilon ,$ then stop. Otherwise, $ d_0=-g_0. $ \\
							\STATE\textbf {Step 2.}  Calculate the stepsize satisfying the improve Wolfe line search \eqref{eq:Daiimprlinsear1} and \eqref{eq:Daiimprlinsear2}.

							\STATE\textbf {Step 3}. Set $ {x_{k + 1}} = {x_k} + {\alpha _k}{d_k}.$ If $\left\| {{g_k}} \right\|_{\infty} \le \varepsilon ,$ then stop.
							
							\STATE\textbf {Step 4}. Update the restart condition.	IterRestart  = IterRestart + 1.
							If $ \left| r_{k-1} \right| \le {\epsilon_1} $ or    $ \left| \bar r_{k-1}  \right| \le \epsilon_1  $\cite{Liu2019An,LiuDaiLiuCGOPT20}, then IterQuad  =  \\
							\quad	\qquad  IterQuad + 1, otherwise
							IterQuad := 0.
							
							\STATE\textbf {Step 5}.  Calculate the search direction. 	
							
							\quad	\textbf{  5.1} Restart.	 If   IterRestart = MaxRestart or (IterQuad = MinQuad and IterQuad  =
							IterRestart), 		set 	$ d_k= - g_k $ and \\
							\quad	\qquad   IterRestart := 0, IterQuad := 0.   Set  $ k=k+1 $ and go to  Step 2. \\	
							\quad	\textbf{  5.2} Compute the search direction $ d_k $ by   \eqref{eq:Truconditon2} with $ \tau_k $ in \eqref{eq:Trutauk}.  Set  $ k=k+1 $ and go to  Step 2. 		
							
						\end{algorithmic}
					\end{algorithm}
					In the SMCG method, IterRestart denote the number of iterations since the last restart. IterQuad  denote the number of continuous iterations such that  $ r_{k-1} $ or $\overline r_{k-1} $ is close to 0.

					
					\section{ Convergence Analysis}
					We will establish the global convergence      of Algorithm 1 for general functions   under   Assumption \ref{Assumption} in the section.
					
					Since Algorithm 1 is restarted with $d_k =-g_k$ at least MaxRestart  iterations,  the global convergence can be obtained easily. So we consider   the global convergence properties of
					Algorithm 1 without the  restart in Step 5.1. In addition, since $ \tau_k $ in    \eqref{eq:Trutauk} chooses adaptively between $ 1 $ and $\tau_k^B $, the   convergence result  based on any one of $ \tau_k $ in \eqref{eq:sometauk}   suffices  to that of Algorithm 1.		
So we   establish the global convergence of the SMCG method \eqref{eq:CGform} and \eqref{eq:Truconditon2} under the Assumption \ref{Assumption}.
					
					According to the   improved Wolfe line search   \eqref{eq:Daiimprlinsear1} and \eqref{eq:Daiimprlinsear2} and Assumption \ref{Assumption} (ii), we  know easily  that \begin{equation}\label{eq:alphabound}\sum\limits_{k = 0}^{ + \infty } { - {\alpha _k}g_k^T{d_k}}  <  + \infty  \;\; \text{and} \;\;   {\alpha _k} \ge \frac{{ - \left( {1 - \sigma } \right)g_k^T{d_k}}}{{L{{\left\| {{d_k}} \right\|}^2}}},\end{equation}
					which implies that
					\begin{equation} \label{Zoutendijk condition}
					\sum\limits_{k = 0}^{ + \infty }  {\frac{{{{\left( {g_k^T{d_k}} \right)}^2}}}{{{{\left\| {{d_k}} \right\|}^2}}}}  <  + \infty .
					\end{equation}
					Together with Lemma \ref{LemmaSuffDesc2}, we obtain
					\begin{equation} \label{eq:Z2condition}
					\sum\limits_{k = 0}^{ + \infty } {\frac{{{{\left\| {{g_k}} \right\|}^4}}}{{{{\left\| {{d_k}} \right\|}^2}}}}  <  + \infty .
					\end{equation}				
					The above inequality is important to analyze the convergence   of the proposed method. 
					
					 The next lemma will be used to  the convergent analysis of  the SMCG method \eqref{eq:CGform} and \eqref{eq:Truconditon2}.
					\begin{lemma}\label{LemmaUkbound}
						Assume $ f $ satisfies Assumption \ref{Assumption},       consider the     subspace minimization conjugate gradient method  \eqref{eq:CGform} and \eqref{eq:Truconditon2} with any one of $ \tau_k $ in \eqref{eq:sometauk}, and $ \alpha_k $ is    calculated by the improved Wolfe line search satisfying \eqref{eq:Daiimprlinsear1} and \eqref{eq:Daiimprlinsear2}. If $\left\| {{g_k}} \right\| \ge \gamma_1  >0$ holds for all $ k\ge 1 $, then 	  	
						\begin{equation} \label{eq:sum uk- uk-1}
							\sum\limits_{k = 0}^\infty  {{{\left\| {{\widetilde u_k} - {\widetilde u_{k - 1}}} \right\|}^2}}  <  + \infty ,
						\end{equation}
						where $   {{\widetilde u_k} = \dfrac{{{d_k}}}{{\left\| {{d_k}} \right\|}}}.   $
					\end{lemma}

					\begin{proof}
						~~We first derive a bound for $  {{u_k}}  $ in \eqref{eq:uk1}.
						By    \eqref{eq:Daiimprlinsear2}, Lemma \ref{LemmaSuffDesc2} and $\left\| {{g_k}} \right\| \ge {\gamma _1}$, we have that
						\begin{equation} \label{eq:Boundykdk}
							y_k^T{d_k} \ge  - \left( {1 - \sigma } \right)g_k^T{d_k} \ge c\left( {1 - \sigma } \right){\left\| {{g_k}} \right\|^2} \ge c\gamma _1^2\left( {1 - \sigma } \right)
						\end{equation}
						and
						\begin{equation} \label{eq:Rboundgk1dk}
							g_{k + 1}^T{d_k} \ge \sigma g_k^T{d_k} = \sigma g_{k + 1}^T{d_k} - \sigma y_k^T{d_k}.
						\end{equation}
						It follows from  Lemma \ref{LemmaSuffDesc2}  that
						\begin{equation} \label{eq:Boundgk1dk}
							g_{k + 1}^T{d_k} = y_k^T{d_k} + g_k^T{d_k} < y_k^T{d_k}.
						\end{equation}
						
						\noindent Combining $\sigma  < 1$, \eqref{eq:Rboundgk1dk}  and \eqref{eq:Boundgk1dk} yields that
						\begin{equation} \label{eq:Boundgk1dkgkdk}
							\frac{{\left| {g_{k + 1}^T{d_k}} \right|}}{{y_k^T{d_k}}} \le \max \left\{ {1,\frac{\sigma }{{1 - \sigma }}} \right\}.
						\end{equation}
						
						According to Assumption \ref{Assumption}, we know that there are two positive constants $ D $ and $  \gamma_2 $ such that
						\begin{equation} \label{eq:Dandgamma2}
						D = \max \left\{ {\left\| {y - z} \right\|:y,z \in {\cal L} = \left\{ {x|f\left( x \right) \le f\left( {{x_0}} \right) + \sum\limits_{k \ge 0} {{{\bar \eta }_k}} } \right\}} \right\},\;\;\left\| {{g_k}} \right\| \le {\gamma _2}.
						\end{equation}
						
						It is  note that ${d_k} \ne 0$  for all $ k\ge 1 $, otherwise Lemma \ref{LemmaSuffDesc}  will imply ${g_k} = 0$. It indicates that $\widetilde u_k$ is well defined.
						Therefore, by using  \eqref{eq:wkparalel}, \eqref{eq:newukvk},  \eqref{eq:Boundgk1dkgkdk},  \eqref{eq:Dandgamma2} and \eqref{eq:Lipschitz continuous}, we  obtain	
				\begin{equation} \label{eq:bounduk}
							\left| {{u_k}} \right|  \le \frac{1}{{1 -\overline \omega_k }}\left( {1 + \left| {\frac{{g_k^T{s_{k - 1}}}}{{s_{k - 1}^T{y_{k - 1}}}}} \right|\frac{{ \left| g_k^T{y_{k - 1} }\right|  }}{{{{\left\| {{g_k}} \right\|}^2}}}} \right)
							\le \frac{1}{{1 - {\xi _1}}}\left( {1 + \max \left\{ {1,\left. {\frac{\sigma }{{1 - \sigma }}} \right\} \left(1+\eta_1 \right) } \right.} \right)
							\buildrel \Delta \over = {{\bar c}_1} > 1.
						\end{equation}

						We  divide $ \bar v_k $ in  \eqref{eq:Tru-vketak}
						into the following  two  parts:
						\begin{equation}\label{v_k^ +}
							v_k^ +  = \max \left\{ {\frac{{g_k^T{y_{k - 1}}}}{{s_{k - 1}^T{y_{k - 1}}}} - \left( {{\tau _k} + \frac{{{{\left\| {{y_k}} \right\|}^2}}}{{s_{k - 1}^T{y_{k - 1}}}}} \right)\frac{{g_k^T{s_{k - 1}}}}{{s_{k - 1}^T{y_{k - 1}}}} + \frac{{g_k^T{s_{k - 1}}{{\left\| {{g_k}} \right\|}^2} - \frac{{g_k^T{y_{k - 1}}{{\left( {g_k^T{s_{k - 1}}} \right)}^2}}}{{s_{k - 1}^T{y_{k - 1}}}}}}{{{{{{\left\| {{g_k}} \right\|}^2}{{\left\| {{s_{k - 1}}} \right\|}^2} - {{\left( {g_k^T{s_{k - 1}}} \right)}^2}}}}} - \eta_k  ,0} \right\},
						\end{equation}
						and
						\begin{equation} \label{eq:v_k^-}
							v_k^ -  =\eta_k = -l_k   \frac{{\left| g_k^T{s_{k - 1}}\right| }}{{{{\left\| {{s_{k - 1}}} \right\|}^2}}},
						\end{equation}
						which satisfy $\bar  v_k=v_k^{+} + v_k ^{-} $.  It follows that    the search direction $ d_k=u_kg_k +   \bar v_k s_{k - 1} $ in \eqref{eq:Truconditon2}       can be  rewritten as
						\[{d_k} = { u_k}{g_k} + \left( {   v_k^{+} +  v_k ^{-} }\right) {s_{k - 1}}.\]

						\noindent Denote
						\begin{equation}\label{eq:omegaanddeltak}{\omega _k} = \frac{{  u_k{g_k} +    {  v_k^ - }{s_{k - 1}}}}{{\left\| {{d_k}} \right\|}},\;\; {\delta _k} = \frac{{   v_k^ + \left\| {{s_{k - 1}}} \right\|}}{{\left\| {{d_k}} \right\|}}. \end{equation}
						Thus, $\widetilde {u}_k$ can be rewritten as
						\begin{equation} \label{{u_k2}}
							{\widetilde u_k} = \frac{{{d_k}}}{{\left\| {{d_k}} \right\|}} = {\omega _k} + {\delta _k}{ \widetilde u_{k - 1}}.
						\end{equation}
						Using the identity $\left\| {{\widetilde u_k}} \right\| = \left\| {{\widetilde u_{k - 1}}} \right\| = 1$, we get that
						\begin{equation} \label{omega _k2}
							\left\| {{\omega _k}} \right\| = \left\| {{\widetilde u_k} - {\delta _k}{\widetilde u_{k - 1}}} \right\| = \left\| {{\delta _k}{\widetilde u_k} - {\widetilde u_{k - 1}}} \right\|.
						\end{equation}
						It follows from ${\delta _k} \ge 0$,  the triangle inequality and \eqref{omega _k2} that
						\begin{equation} \label{u_k-u_{k - 1}1}
							\begin{aligned}
								\left\| {{\widetilde u_k} - {\widetilde u_{k - 1}}} \right\|
								&\le \left\| {\left( {1 + {\delta _k}} \right){\widetilde u_k} - \left( {1 + {\delta _k}} \right){\widetilde u_{k - 1}}} \right\|\\
								& \le \left\| {{\widetilde u_k} - {\delta _k}{\widetilde u_{k - 1}}} \right\| + \left\| {{\delta _k}{\widetilde u_k} - {\widetilde u_{k - 1}}} \right\|\\
								& = 2\left\| {\omega _k}\right\|.
							\end{aligned}
						\end{equation}
						
						\noindent By \eqref{eq:Truconditon2}, \eqref{eq:etakk}, \eqref{eq:bounduk} and \eqref{eq:v_k^-}, we can obtain that
						\begin{equation} \label{eq:uvgs}
							\left\| {{   u_k}{g_k} +    v_k^ - {s_{k - 1}}} \right\| \le { \left|   u_k \right|}\left\| {{g_k}} \right\| + \left| {v_k^ - } \right|\left\| {{s_{k - 1}}} \right\| \le   \left( {{ \bar c_1} + \bar \xi_2 } \right) \left\| {{g_k}} \right\| .
						\end{equation}

						Combining  \eqref{u_k-u_{k - 1}1},   \eqref{eq:omegaanddeltak} and \eqref{eq:uvgs} yields
						\begin{equation} \label{eq:uk-uk-1Inequa}
							\left\| {{\widetilde u_k} - {\widetilde u_{k - 1}}} \right\| \le 2\left\| {{\omega _k}} \right\| \le 2\left( {{ \bar c_1} + \bar \xi_2 } \right)\frac{{\left\| {{g_k}} \right\|}}{{\left\| {{d_k}} \right\|}}.
						\end{equation}
						Similarly, for the search direction $ d_k= -g_k $ in   \eqref{eq:Truconditon2}, we can easily obtain   \eqref{eq:uk-uk-1Inequa} by setting $ u_k=-1$, $v_k^{+} = v_k^{-} =0$ in  \eqref{eq:omegaanddeltak} due to $ \bar c_1 >1 $.
						Therefore, together with \eqref{eq:Z2condition} and  $\left\| {{g_k}} \right\| \ge \gamma_1 , $ we have
						\begin{equation} \label{sum limits}
							\sum\limits_{k = 0}^\infty  {{{\left\| {{\widetilde u_k} - {\widetilde u_{k - 1}}} \right\|}^2}}  \le \frac{{4{{\left( {{ \bar c_1} + \bar \xi_2 } \right)}^2}}}{{\gamma _1^2}}\sum\limits_{k = 0}^\infty  {\frac{{{{\left\| {{g_k}} \right\|}^4}}}{{{{\left\| {{d_k}} \right\|}^2}}}}  <  + \infty,
						\end{equation}
						which completes the proof. \qed
					\end{proof}
					
					The global convergence  is established under Assumption \ref{Assumption} in the following theorem.

					\begin{theorem} \label{thm:31}
						Assume $ f $ satisfies Assumption \ref{Assumption},    consider the     subspace minization conjugate gradient method  \eqref{eq:CGform} and \eqref{eq:Truconditon2} with any one of $ \tau_k $ in \eqref{eq:sometauk}, and $ \alpha_k $ is    calculated by the improved Wolfe line search satisfying \eqref{eq:Daiimprlinsear1} and \eqref{eq:Daiimprlinsear2}.      Then,
						\begin{equation} \label{eq:liminf}
							\mathop {\lim \inf }\limits_{k \to \infty } \left\| {{g_k}} \right\| = 0.
						\end{equation}
					\end{theorem}
					
					\begin{proof}

						~~We proceed it  by contradiction, namely,  suppose that $\left\| {{g_k}} \right\| \ge {\gamma _1}, \;\text{ where}\;   {{\gamma _1} > 0}  $, for all $k \ge 0$. 		
						By Lemma \ref{LemmaSuffDesc2} and  Cauchy-Schwarz   inequality, we have that
						\[\left\| {{d_{k - 1}}} \right\| \ge {   c}\left\| {{g_{k - 1}}} \right\| \ge {   c}{\gamma _1},\]
						which together with Assumption \ref{Assumption} (i) yields
						\[  \left\| {{s_{k - 1}}} \right\| = {\alpha _{k - 1}}\left\| {{d_{k - 1}}} \right\| \le \frac{D}{{{   c}{\gamma _1}}}\left\| {{d_{k - 1}}} \right\|,\]
						where $ D $ is    given by \eqref{eq:Dandgamma2}.
						
						As a result,   for the choices of $ \tau_k $  in \eqref{eq:sometauk}, it is not difficult to obtain from \eqref{eq:Boundykdk} that
						\begin{equation} \label{eq:taukbound1}
							{\tau _k^{(1)} }= 1 = \frac{1}{{{   c}{\gamma _1}}}{   c}{\gamma _1} \le \frac{1}{{{   c}{\gamma _1}}}\left\| {{d_{k - 1}}} \right\|
						\end{equation}
						and
						\begin{equation}\label{eq:taukbound2}
							\tau _k^B = \frac{{s_{k - 1}^T{y_{k - 1}}}}{{{{\left\| {{s_{k - 1}}} \right\|}^2}}}  \le \tau _k^H = \frac{{{{\left\| {{y_{k - 1}}} \right\|}^2}}}{{s_{k - 1}^T{y_{k - 1}}}} \le \frac{{L^2D}}{{c\gamma _1^2\left( {1 - \sigma } \right)}}\left\| {{d_{k - 1}}} \right\|.
						\end{equation}
						
						The following  is divided into the   three steps.


						(i) A bound for $  {{v_k}}  $ and $\eta_k   $ in \eqref{eq:Truconditon2}.
							
								By  \eqref{eq:Truconditon2}, \eqref{eq:Boundgk1dkgkdk}, \eqref{eq:taukbound1}, \eqref{eq:taukbound2}, \eqref{eq:Lipschitz continuous}, \eqref{eq:Dandgamma2} and ${\gamma _1} \le \left\| {{g_k}} \right\| \le {\gamma _2}$, we have
							
							\begin{equation}\label{eq:|bk|}	
								\begin{aligned}
									\left| {{v_k}} \right|
									&
									= \left| {\frac{{1 - 2{\omega _k}}}{{1 - {\omega _k}}}\frac{{g_k^T{y_{k - 1}}}}{{s_{k - 1}^T{y_{k - 1}}}} - \left( {{\tau _k} + \frac{{{{\left\| {{y_{k - 1}}} \right\|}^2}}}{{s_{k - 1}^T{y_{k - 1}}}}} \right)\frac{{g_k^T{s_{k - 1}}}}{{s_{k - 1}^T{y_{k - 1}}}} + \frac{1}{{1 - {\bar\omega _k}}}\frac{{g_k^T{s_{k - 1}}}}{{{{\left\| {{s_{k - 1}}} \right\|}^2}}}} \right|\\&
									\le \frac{1}{{1 - {\xi _1}}}\frac{{{\gamma _2}L}}{{c\gamma _1^2\left( {1 - \sigma } \right)}}\left\| {{d_{k - 1}}} \right\| + \max \left\{ {1,\frac{\sigma }{{1 - \sigma }}} \right\}\left( {\max \left\{ {\frac{1}{{c{\gamma _1}}},\frac{{L^2D}}{{c\gamma _1^2\left( {1 - \sigma } \right)}}} \right\} + \frac{{L^2D}}{{c\gamma _1^2\left( {1 - \sigma } \right)}}} \right)\left\| {{d_{k - 1}}} \right\|\\&
									\;\;\;\;\;\;\;\;\;\;\;\;\;\; + \frac{1}{{1 - {\xi _1}}}\frac{{L{\gamma _2}}}{{\left( {1 - \sigma } \right)c\gamma _1^2}}\left\| {{d_{k - 1}}} \right\|\\&
									\le \left[ {\frac{2}{{1 - {\xi _1}}}\frac{{{\gamma _2}L}}{{c\gamma _1^2\left( {1 - \sigma } \right)}}{\rm{ + }}\max \left\{ {1,\frac{\sigma }{{1 - \sigma }}} \right\}\left( {\max \left\{ {\frac{1}{{c{\gamma _1}}},\frac{{L^2D}}{{c\gamma _1^2\left( {1 - \sigma } \right)}}} \right\} + \frac{{L^2D}}{{c\gamma _1^2\left( {1 - \sigma } \right)}}} \right)} \right]\left\| {{d_{k - 1}}} \right\|\\&
									\buildrel \Delta \over = {{\bar c}_2}\left\| {{d_{k - 1}}} \right\|
								\end{aligned}
							\end{equation}
							and		
							\begin{equation}
							\begin{aligned}
							\label{eq:etabound}\left|\eta_k \right|
							&
							\le  \frac{{l_k\left\| {{g_k}} \right\|\left\| {{s_{k - 1}}} \right\|}}{{{{\left\| {{s_{k - 1}}} \right\|}^2}}}
							\le {\bar \xi_2\left\| {{d_{k-1}}}\right\|}\frac{\left\| {{g_k}}\right\|}{\left\| {{s_{k-1}}} \right\|\left\|{{d_{k-1}}} \right\|}
							\le L{\bar \xi_2\left\| {{d_{k-1}}}\right\|}\frac{\left\| {{g_k}}\right\|}{\left\| {{y_{k-1}}} \right\|\left\|{{d_{k-1}}} \right\|}\\&
							\le L{\bar \xi_2}\frac{\gamma_2}{{c\gamma _1^2\left( {1 - \sigma } \right)}} \left\| {{d_{k-1}}}\right\|
							\buildrel \Delta \over = {{\bar{\bar{c}}_2}}\left\| {{d_{k - 1}}} \right\|
							\end{aligned} 												
						 \end{equation}

						(ii)
							A bound on the steps ${s_k}$.
							
							For any $l \ge k$, by the definition of ${\widetilde u_k}$  in Lemma \ref{LemmaUkbound} we have
							\begin{equation} \label{eq:{x_l} - {x_k}}
								{x_l} - {x_k} = \sum\limits_{j = k}^{l - 1} {  \left( {{x_{j + 1}} - {x_j}} \right)  = } \sum\limits_{j = k}^{l - 1} {\left\| {{s_j}} \right\|} {\widetilde u_j} = \sum\limits_{j = k}^{l - 1} {\left\| {{s_j}} \right\|} {\widetilde u_k} + \sum\limits_{j = k}^{l - 1} {\left\| {{s_j}} \right\|} \left( {{\widetilde u_j} - {\widetilde u_k}} \right),
							\end{equation}
							which yields that
							\begin{equation} \label{eq:sum sj1}
								\sum\limits_{j = k}^{l - 1} {\left\| {{s_j}} \right\|}  \le \left\| {{x_l} - {x_k}} \right\| + \sum\limits_{j = k}^{l - 1} {\left\| {{s_j}} \right\|} \left\| {{\widetilde u_j} - {\widetilde u_k}} \right\| \le D + \sum\limits_{j = k}^{l - 1} {\left\| {{s_j}} \right\|} \left\| {{\widetilde u_j} - {\widetilde u_k}} \right\|.
							\end{equation}

							Let $\Delta $ be any positive integer such  that
							\begin{equation} \label{eq:Delta4cD}
								\Delta  \ge 4\bar c_2D.
							\end{equation}
							
					\noindent		According to Lemma \ref{LemmaUkbound}, we can choose ${k_0} >0$   such that
							\begin{equation} \label{eq:sum ui+1-ui}
								\sum\limits_{i \ge {k_0}} {{{\left\| {{\widetilde u_{i + 1}} - {\widetilde u_i}} \right\|}^2}}  \le \frac{1}{{4\Delta }}.
							\end{equation}
							
					\noindent		If $j > k \ge {k_0}$ and $j - k \le \Delta $, then we know from \eqref{eq:sum ui+1-ui} and the Cauchy-Schwarz inequality that
							\begin{equation} \label{eq:uj-uk}
								\left\| {{\widetilde u_j} - {\widetilde u_k}} \right\| \le \sum\limits_{i = k}^{j - 1} {\left\| {{\widetilde u_{i + 1}} - {\widetilde u_i}} \right\|}  \le \sqrt {j - k} {\left( {\sum\limits_{i = k}^{j - 1} {{{\left\| {{\widetilde u_{i + 1}} - {\widetilde u_i}} \right\|}^2}} } \right)^{{1 \mathord{\left/
												{\vphantom {1 2}} \right.
												\kern-\nulldelimiterspace} 2}}} \le \sqrt \Delta  {\left( {\frac{1}{{4\Delta }}} \right)^{{1 \mathord{\left/
												{\vphantom {1 2}} \right.
												\kern-\nulldelimiterspace} 2}}} = \frac{1}{2}.
							\end{equation}
							
					\noindent		Combining \eqref{eq:uj-uk} with \eqref{eq:sum sj1} implies
							\begin{equation} \label{eq:sumsj2}
								\sum\limits_{j = k}^{l - 1} {\left\| {{s_j}} \right\|}  \le 2D,
							\end{equation}
							
				\noindent			where $l > k \ge {k_0}$ and $l - k \le \Delta $.
							
						(iii)
							A bound on the directions $d_l$.
							
							For  the search direction \eqref{eq:Truconditon2},
							by  \eqref{eq:bounduk}, \eqref{eq:etabound} and \eqref{eq:|bk|}, we have
							\begin{equation} \label{eq:dkkkk}
								{\left\| {{d_l}} \right\|^2} \le  \left(  u_l   \left\| {{g_l}} \right\| + \max \left\{ { \left| {{\eta _l}} \right|,\left| {{v_l}} \right|} \right\}\left\| {{s_{l - 1}}} \right\|  \right)^2  \le 2\bar c_1^2\gamma _2^2 + 2\bar c_3^2{\left\| {{s_{l - 1}}} \right\|^2}{\left\| {{d_{l - 1}}} \right\|^2},
							\end{equation}
							where $\bar c_3=\max\left\lbrace \bar c_2,  {\bar {\bar {c}}_2 }\right\rbrace$, $\bar c_1 $ and $\bar c_2 $ are given by \eqref {eq:bounduk} and \eqref{eq:|bk|}, respectively. Let ${S_i} = 2\bar c_3^2{\left\| {{s_i}} \right\|^2}$,   for any $l > {k_0}$, we have
							\begin{equation} \label{eq:dl}
								{\left\| {{d_l}} \right\|^2} \le 2 \bar c_1^2\gamma _2^2\left( {\sum\limits_{i = {k_0} + 1}^l {\prod\limits_{j = i}^{l - 1} {{S_j}} } } \right) + {\left\| {{d_{{k_0}}}} \right\|^2}\prod\limits_{j = {k_0}}^{l - 1} {{S_j}} .
							\end{equation}
							Note that the product is define to be 1 whenever the index range is vacuous. Now we derive   a product of $\Delta$ consecutive $S_j$ by  the arithmetic-geometric mean inequality,  \eqref{eq:sumsj2} and \eqref{eq:Delta4cD} for any ${k \ge {k_0}}$:
							\begin{equation} \label{eq:prod Sj}
								\begin{aligned}
									\prod\limits_{j = k}^{k + \Delta  - 1} {{S_j}}
									&= \prod\limits_{j = k}^{k + \Delta  - 1} {2{\bar c_3^2}{{\left\| {{s_j}} \right\|}^2}}  = {\left( {\prod\limits_{j = k}^{k + \Delta  - 1} {\sqrt 2 \bar c_3\left\| {{s_j}} \right\|} } \right)^2}\\
									& \le {\left( {\frac{{\sum\limits_{j = k}^{k + \Delta  - 1} {\sqrt 2  \bar c_3\left\| {{s_j}} \right\|} }}{\Delta }} \right)^{2\Delta }} \le {\left( {\frac{{2\sqrt 2  \bar c_3D}}{\Delta }} \right)^{2\Delta }} \le \frac{1}{{{2^\Delta }}}.
								\end{aligned}
							\end{equation}
							
					\noindent		As a result, there must exist a positive constant $  c_3 >0 $ such that  $\left\| {{d_l}} \right\| \le {	  c_3}$. Combining   $ \left\| {{g_k}} \right\| \ge \gamma_1 $ with $\left\| {{d_l}} \right\| \le {	  c_3}$ yields
							$$\sum\limits_{k = 0}^{ + \infty } {\frac{{{{\left\| {{g_k}} \right\|}^4}}}{{{{\left\| {{d_k}} \right\|}^2}}}} = +\infty.	$$

							%
							%
							%
					\noindent		which contradicts \eqref{eq:Z2condition}. Therefore, we obtain  \eqref{eq:liminf}, which completes the proof. \qed
							
					\end{proof}

				\begin{rem}

				It is not difficult to see that if $ f $ is convex, then we can obtain the strong convent result:
				$$\mathop {\lim  }\limits_{k \to \infty } \left\| {{g_k}} \right\| = 0.$$
						\end{rem}

			 	It follows from \eqref{eq:Daiimprlinsear1} that $	f_{k+1} - f_k   \le      {\bar \eta _k}     $, which together with $  \sum\limits_{k \ge 0} {{{\overline \eta  }_k}}< +\infty  $ and Assumption  \ref{Assumption}   (ii)   implies  that the sequence $ \left\lbrace f_k \right\rbrace  $  is convergent. We denoted its limite by $ f^* $.  It also can deduce from Theorem \ref{thm:31} and Assumption  \ref{Assumption}   (ii) that   there exists a convergent subquece $ \left\lbrace x_{k_i} \right\rbrace  $ of $ \left\lbrace x_{k } \right\rbrace  $ such that  $ g\left( \hat{x}    \right) = 0 $, where  $ x_{k_j} \to \hat{x} $ when $j\rightarrow   +\infty $.  Since $ f $ is convex on $ \mathbb{R}^n $, we have $ f^*=f(\hat{x}) \le f(x), \; \forall x \in  \mathbb{R}^n$. 	If  there exists another convergent subquece $ x_{k_j} $ of  $ x_{k} $  such that  $ g\left( \bar x \right) \ne 0 $, where  $ x_{k_j} \to \bar x $ when $j\rightarrow   +\infty $. Since   $ f(\bar x) =f^* \le f(x), \; \forall x \in  \mathbb{R}^n$, we know that $ \bar x $ is a global minimizer, which contradicts $ g\left( \bar x \right) \ne 0 $. Therefore, all accumulations  of $ \left\lbrace x_{k }\right\rbrace  $ are stationary points, which implies that $\mathop {\lim  }\limits_{k \to \infty } \left\| {{g_k}} \right\| = 0.$   

					\section{Numerical experiments}

					We  compare Algorithm 1  with    CGOPT (1.0) \cite{Dai2013A},  SMCG$ \_ $BB \cite{Liu2019An} and CG$ \_ $DESCENT (5.3) \cite{Hager2005A} in the section. It is widely accepted  that CGOPT   and CG$ \_ $DESCENT are the most two famous conjugate gradient software packages. Algorithm 1   was implemented based on the C code of CGOPT (1.0), which is available from Dai's homepage: \url{http://lsec.cc.ac.cn/~dyh/software.html}.  The test collection includes   147  unconstrained optimization problems   from the CUTEst library \cite{Gould2001CUTEr}, which can be found in \cite{BBQ-HDL},  and the  dimensions of the 147 test problems and the initial points are all default. The codes of CG$ \_ $DESCENT (5.3) and SMCG$ \_ $BB can be found in Hager's homepage: \url{http://users.clas.ufl.edu/hager/papers/Software} and \url{https://www.scholat.com/liuzexian}, respectively. In Algorithm 1,  we take the following  parameters:
					\[\xi_1= 0.75, \; \xi_2 =0.5,\; \bar{\bar\xi}_2=0.2,
					\;{\bar\xi}_2=10,\;\eta_1=0.99,\]
					\[\eta_2=3,\;\xi_3=7.5\times10^{-5},
					\xi_4=9\times 10^{-4},\;\xi_5=0.9, \;\xi_6=10  \]
					and the  other parameters  used the default value in CGOPT (1.0).    CGOPT   (1.0), CG$ \_ $DESCENT (5.3) and  SMCG$ \_ $BB  used  all  default values of these parameter 	   in their codes     but  the stopping condition. Especially, CG$ \_ $DESCENT (5.3)  used the  default line search---the combination of the original Wolfe conditions and the approximate
					Wolfe conditions, which  performed very
					well in the numerical tests. 
					All test methods are terminated    if    $\parallel {g_{k}}{\parallel _\infty } \le {10^{ - 6}}$ is satisfied.

						The performance profiles introduced by Dolan and Mor\'e \cite{Dolan2002Benchmarking} are used to display the performances of these test algorithms. 	In the following  figures,  ``$ N_{iter} $'', ``$ N_f $'',  ``$ N_g $'' and ``$ T_{cpu} $'' represent the number of iterations, the number of function evaluations, the number of gradient  evaluations and CPU time (s), respectively.
						
						
						The numerical experiments are divided into the following three groups.

			In the first group of 	the numerical experiments, we test the numerical performance of Algorithm 1 with different  $\tau_k$ in \eqref{eq:Trutauk} and \eqref{eq:sometauk}. The default $\tau_k$ in Algorithm 1  is given by \eqref{eq:Trutauk}. Figures \ref{fig:iterDifftauk}-\ref{fig:NgDifftauk} present the numerical performance in term of the number of iterations, the number of function evaluations and  the number of gradient  evaluations. We do not test the performance about the running time since it is  similar to the above  figures.  As observed in the Figures \ref{fig:iterDifftauk}-\ref{fig:NgDifftauk},    $\tau_k$  in \eqref{eq:Trutauk} is the most efficient  for Algorithm 1,  followed by $\tau_k^B$, and  $\tau_k^H$ is the worst.  								
					
					\begin{figure}[htbp]	
						\begin{minipage}[t]{0.5\linewidth}
							\centering
							\includegraphics[width=9.3cm,height=7cm]{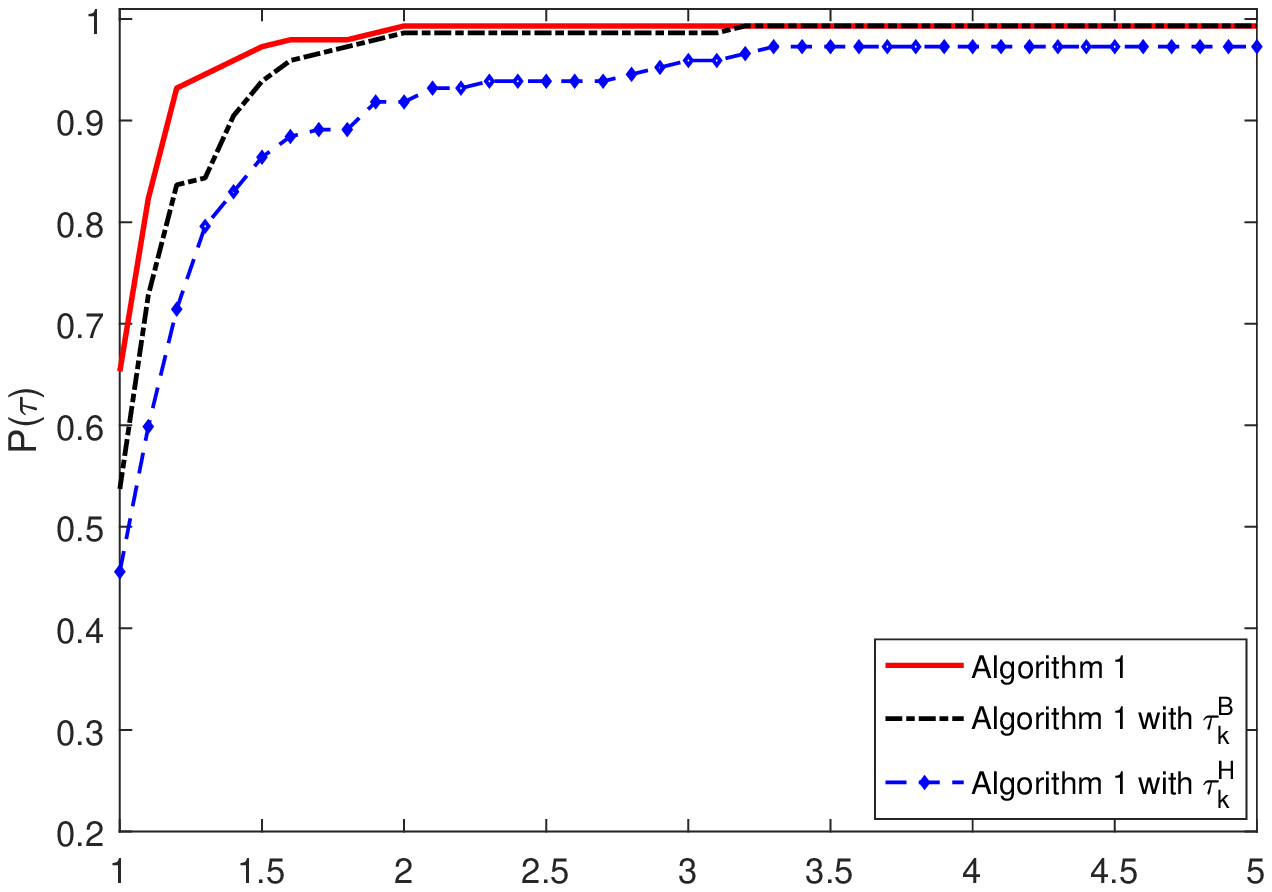}
							\caption{$ N_{iter}$  }
							\label{fig:iterDifftauk}
						\end{minipage}%
						\begin{minipage}[t]{0.5\linewidth}
							\centering
							\includegraphics[width=9.3cm,height=7cm]{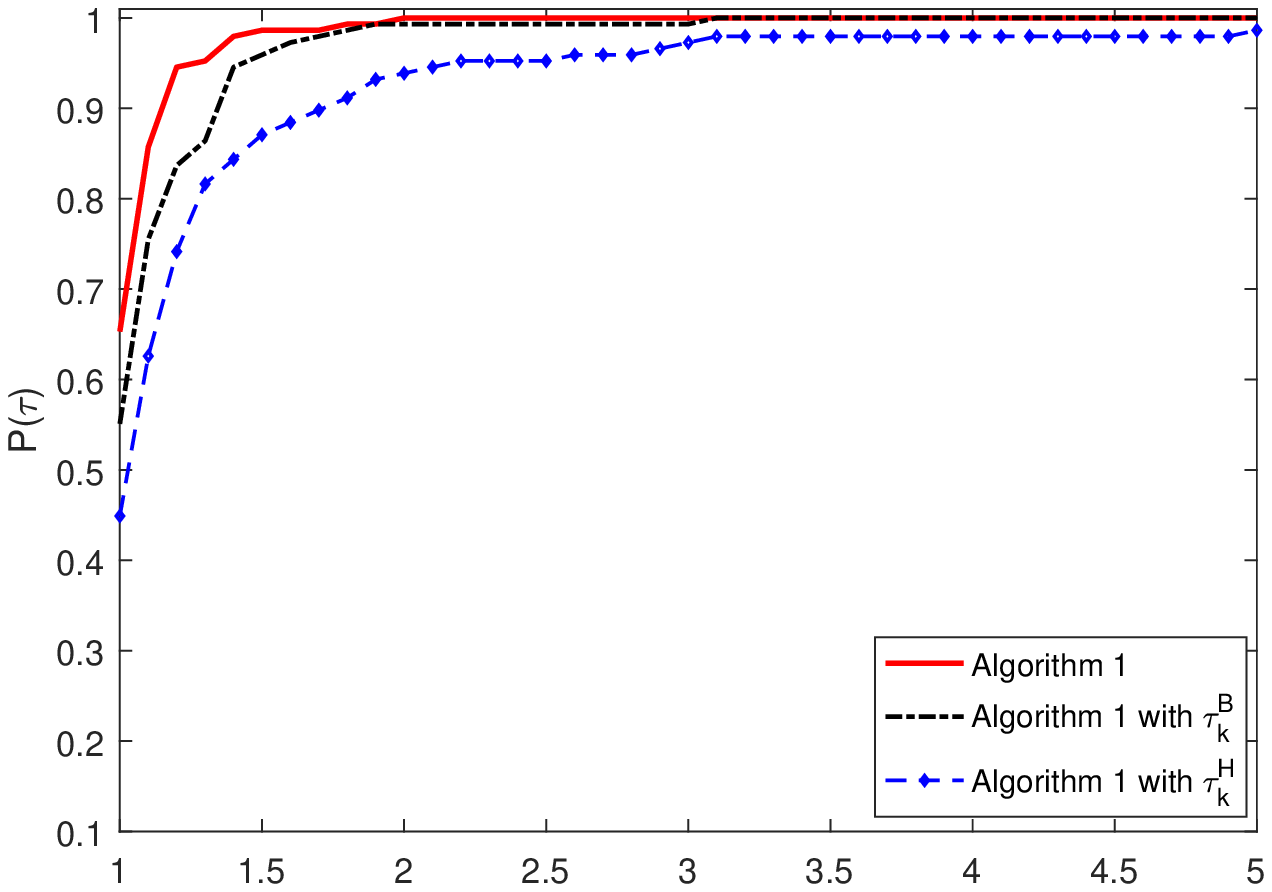}
							\caption{   $ N_{f}$  }
							\label{fig:NfDifftauk}
						\end{minipage}%
					\end{figure}

					\begin{figure}[htbp]	
						\begin{minipage}[t]{0.5\linewidth}
							\centering
							\includegraphics[width=9.3cm,height=7cm]{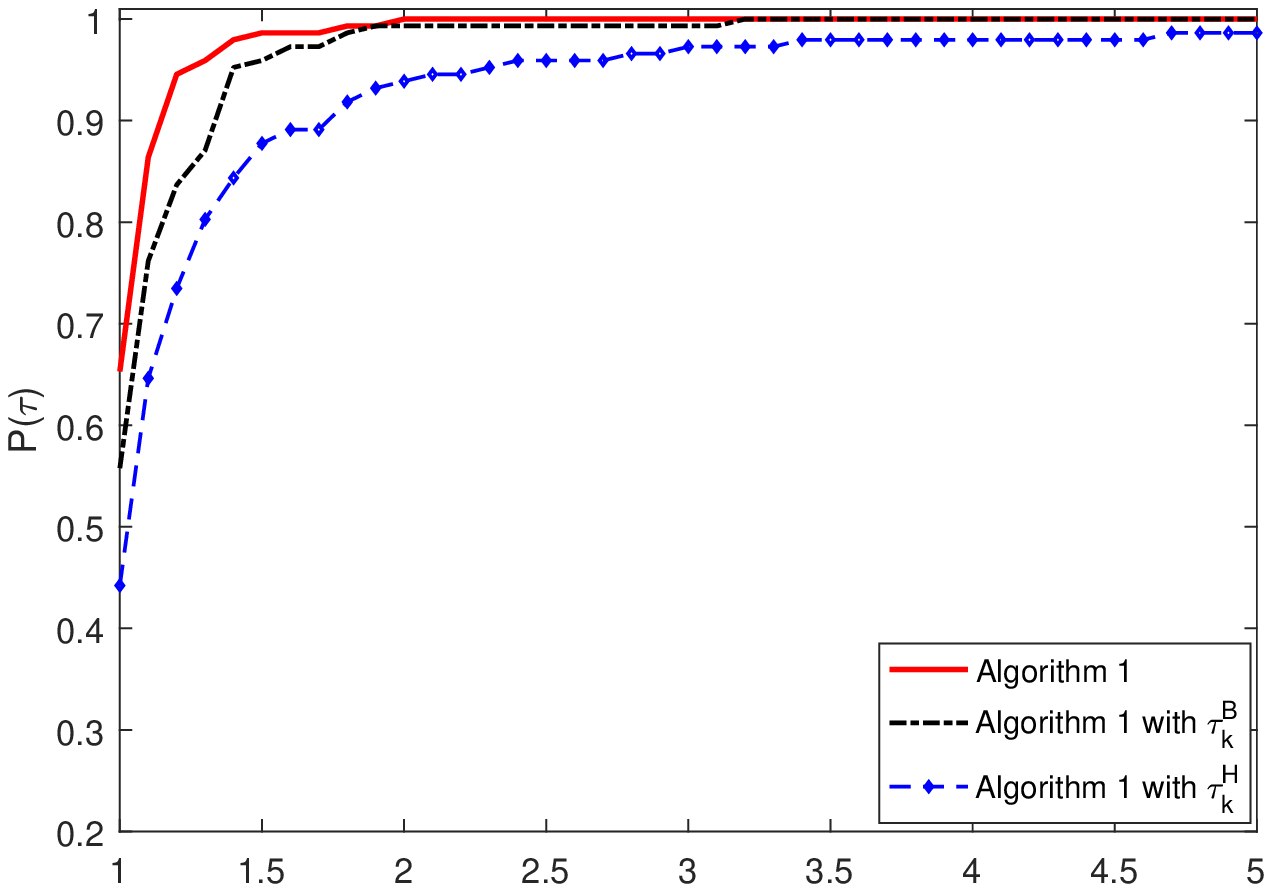}
							\caption{$ N_{g}$  }
							\label{fig:NgDifftauk}
						\end{minipage}%
						\begin{minipage}[t]{0.5\linewidth}
							\centering
							\includegraphics[width=9.3cm,height=7cm]{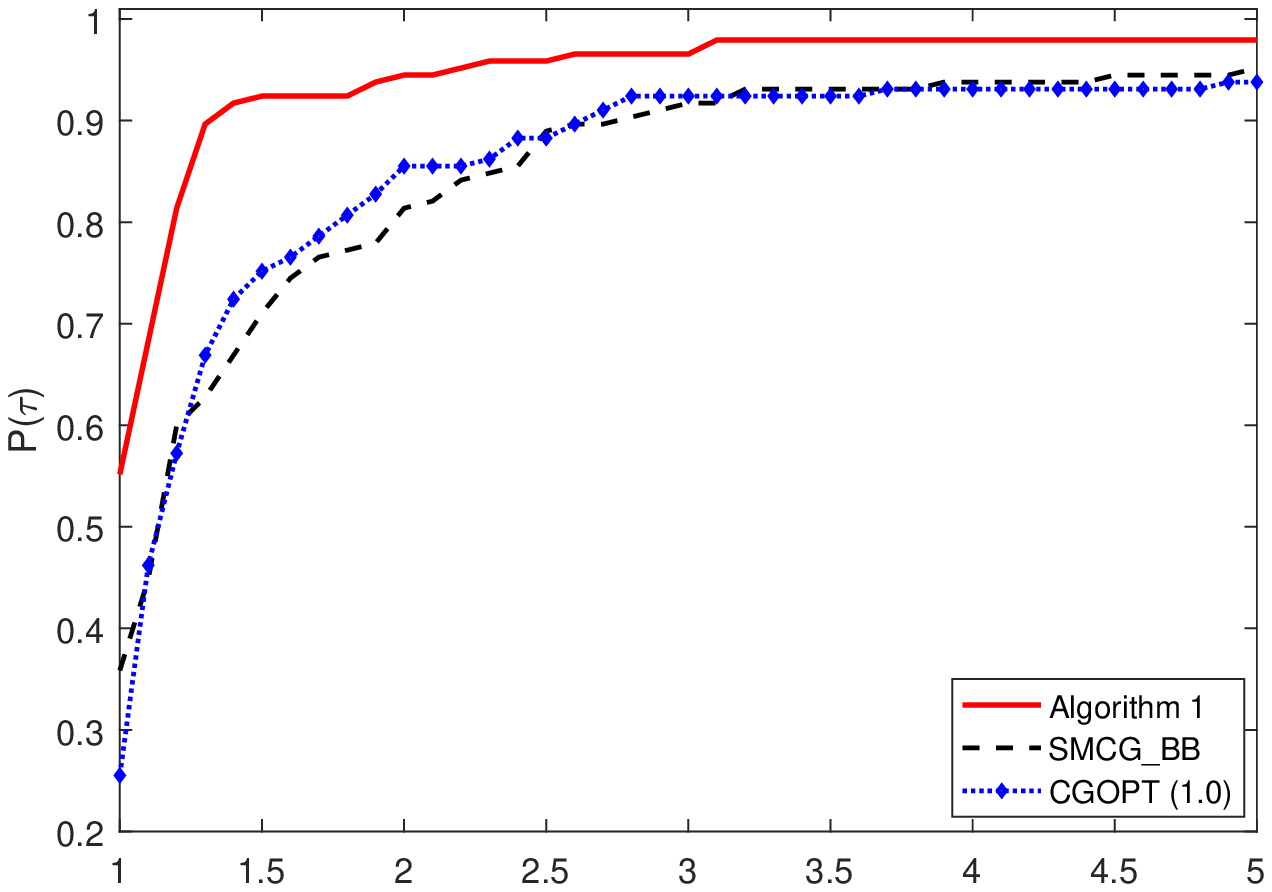}
							\caption{   $ N_{iter}$  }
							\label{fig:IterA1CGoptSMCGBB}
						\end{minipage}%
					\end{figure}

					\begin{figure}[htbp]	
						\begin{minipage}[t]{0.5\linewidth}
							\centering
							\includegraphics[width=9.3cm,height=7cm]{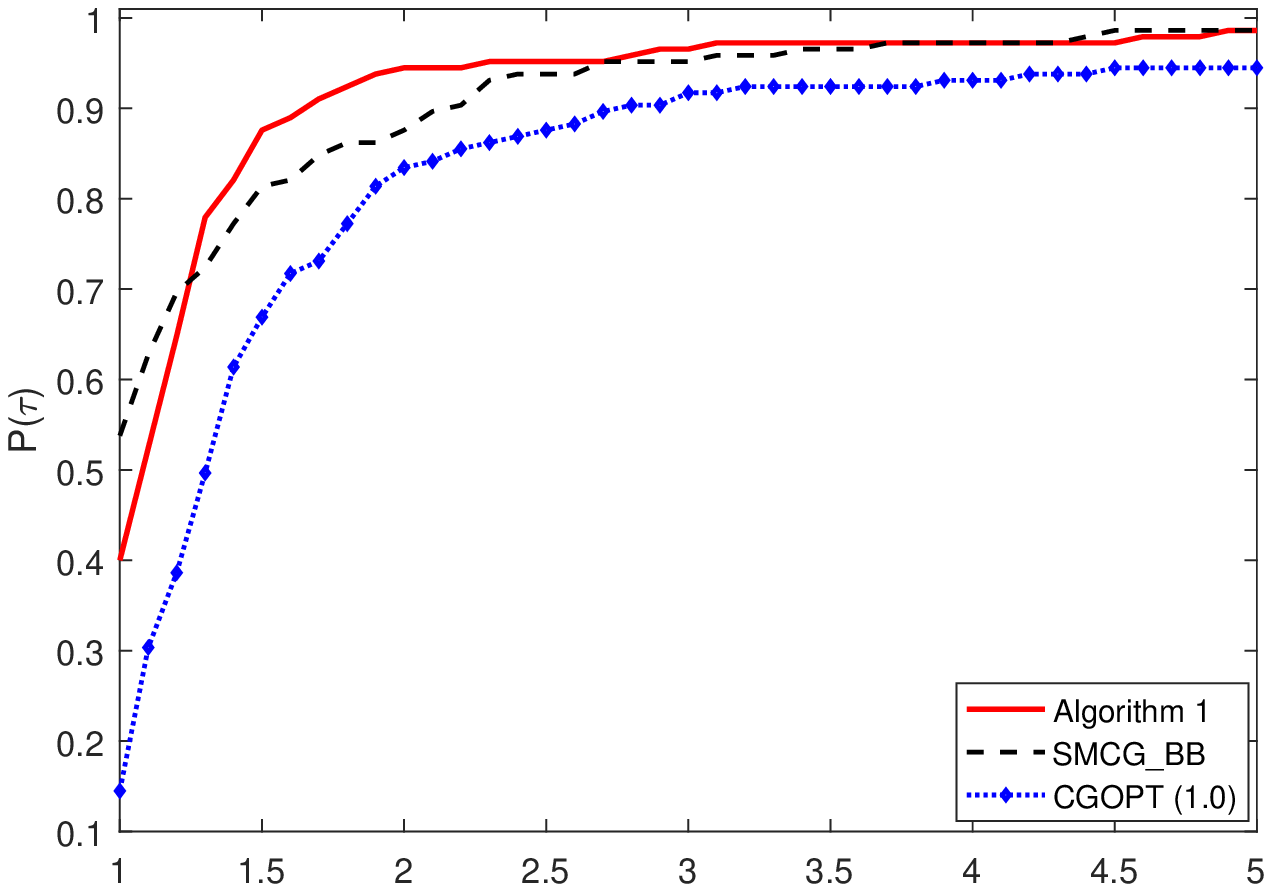}
							\caption{$ N_{f}$  }
							\label{fig:NfA1CGoptSMCGBB}
						\end{minipage}%
						\begin{minipage}[t]{0.5\linewidth}
							\centering
							\includegraphics[width=9.3cm,height=7cm]{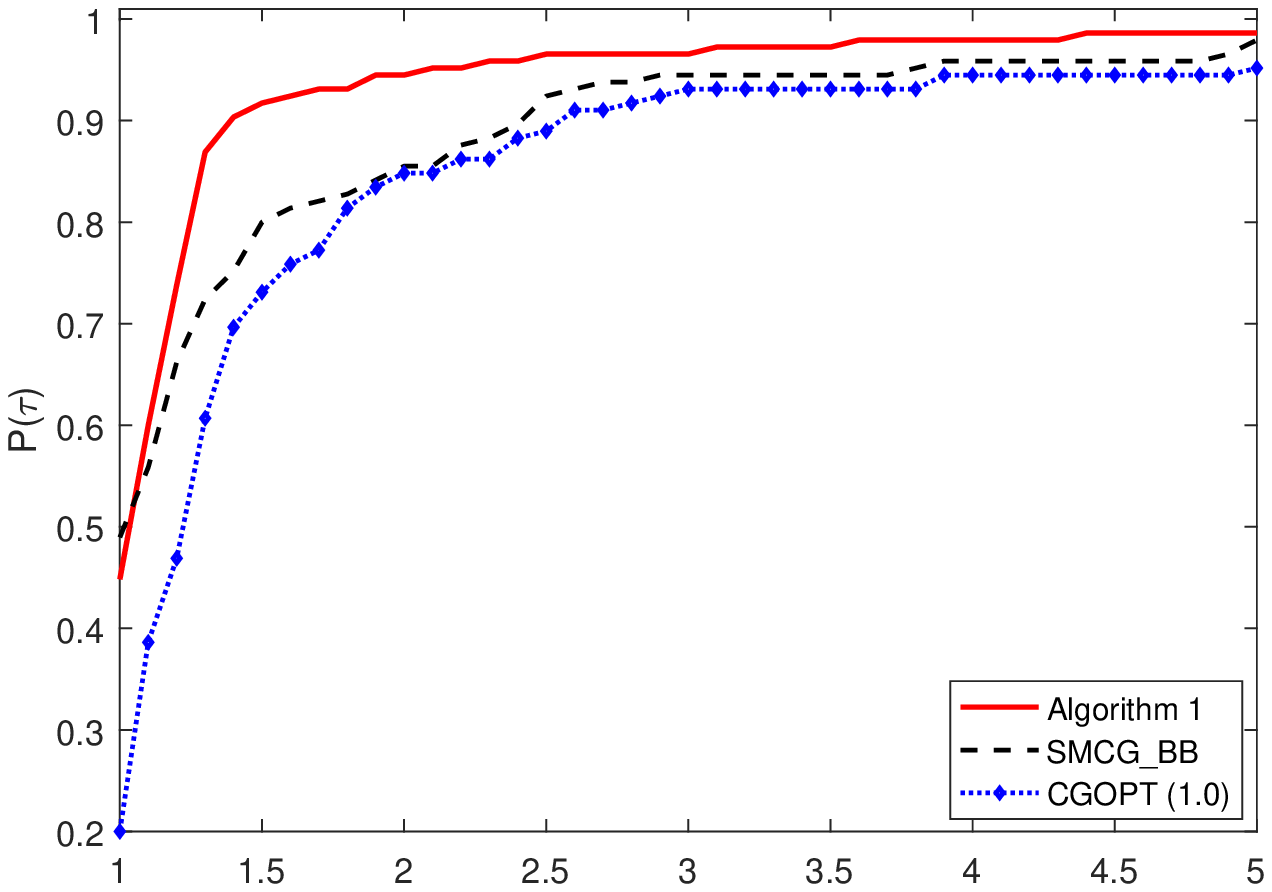}
							\caption{   $ N_{g}$  }
							\label{fig:NgA1CGoptSMCGBB}
						\end{minipage}%
					\end{figure}	
					
					\begin{figure}[htbp]	
						\begin{minipage}[t]{0.5\linewidth}
							\centering
							\includegraphics[width=9.3cm,height=7cm]{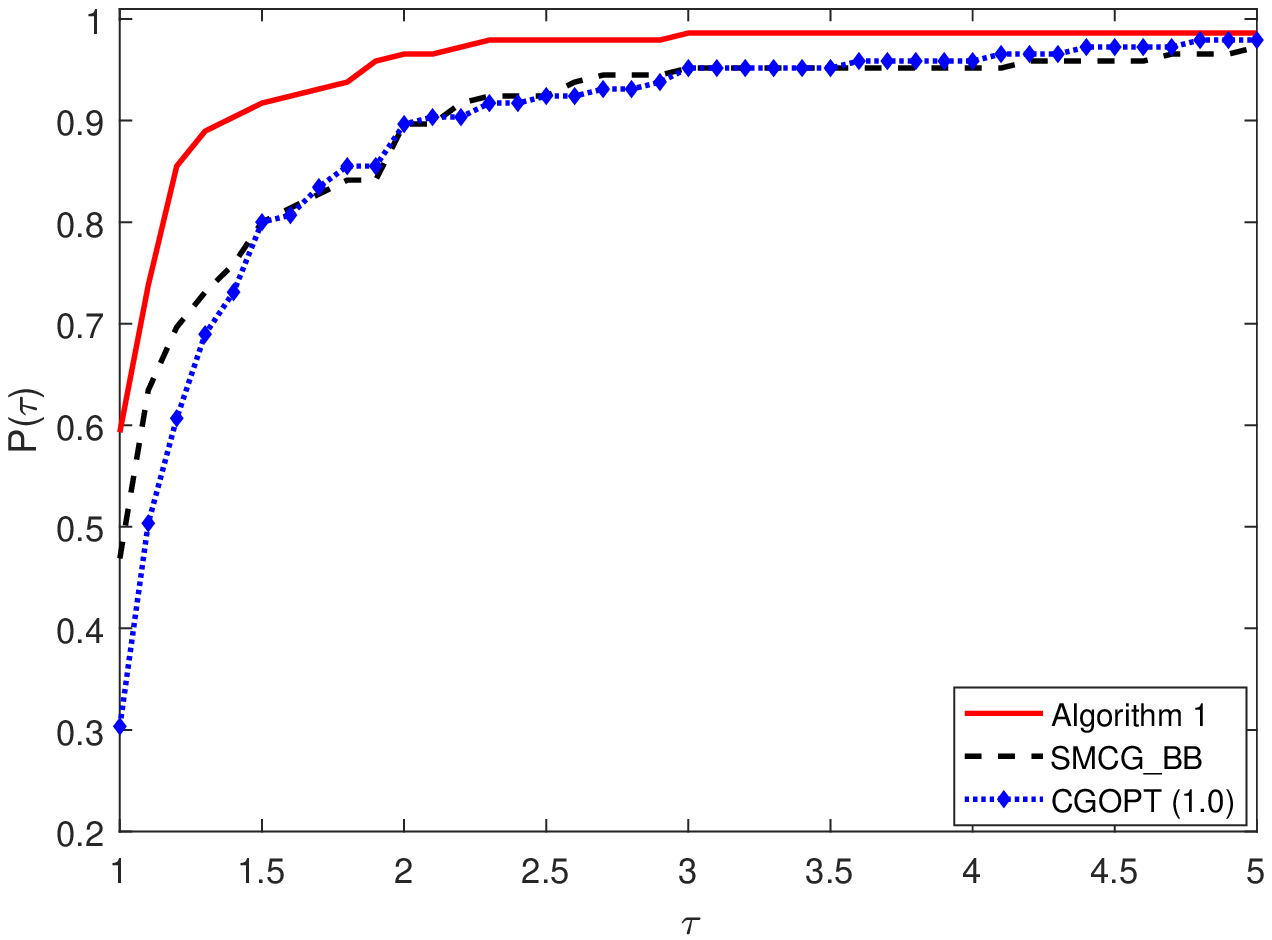}
							\caption{$T_{cpu}  $  }
							\label{fig:TcpuA1CGoptSMCGBB}
						\end{minipage}%
						\begin{minipage}[t]{0.5\linewidth}
							\centering
							\includegraphics[width=9.3cm,height=7cm]{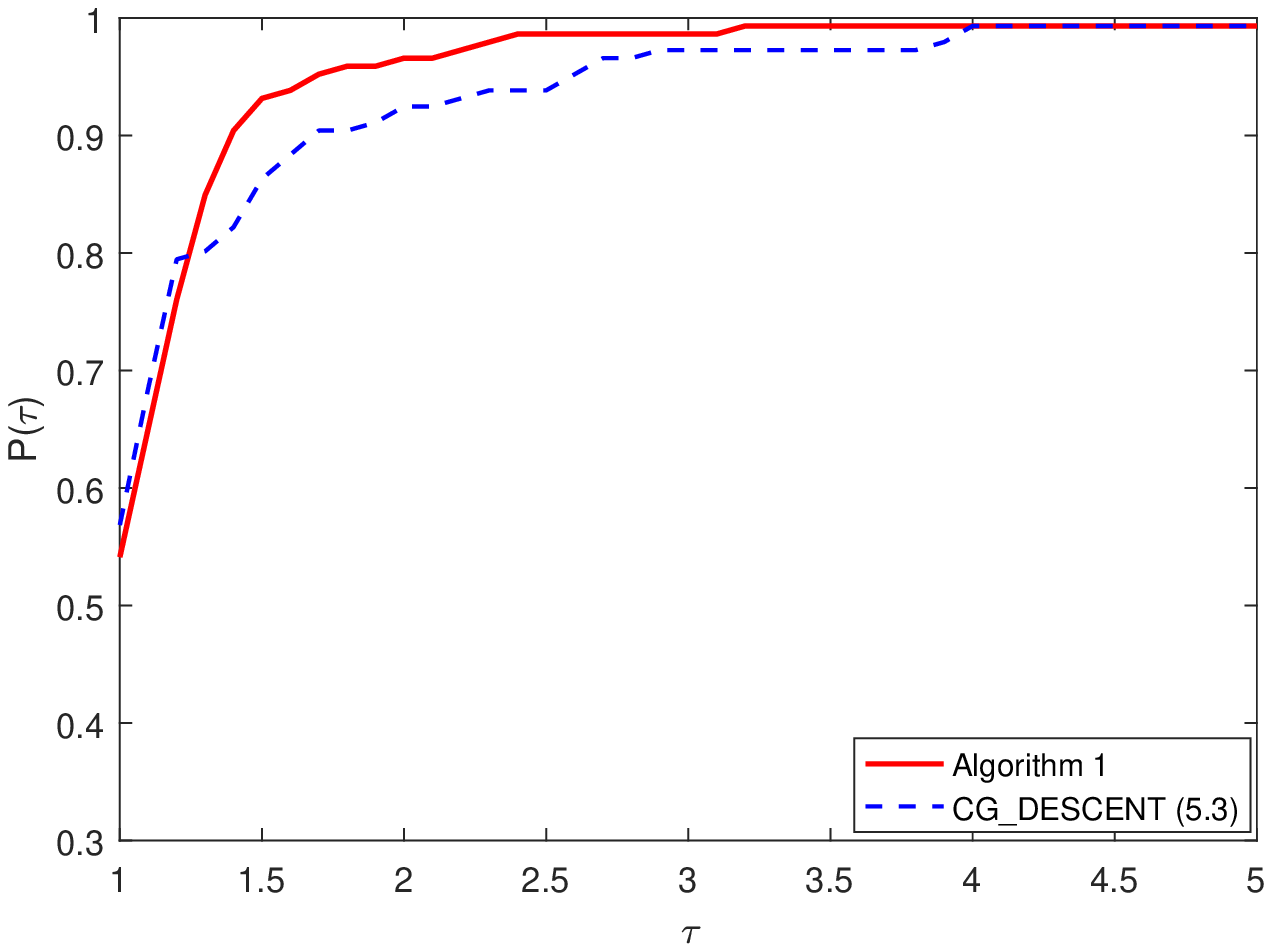}
							\caption{   $N_{iter}  $  }
							\label{fig:IterA1CG53}
						\end{minipage}%
					\end{figure}

					In the second group of 	the numerical experiments, we compare the   performance of Algorithm 1 with that of SMCG$\_$BB and CGOPT (1.0).   As shown in Figure \ref{fig:IterA1CGoptSMCGBB}, we observe that the Algorithm 1 performs much better than SMCG$\_$BB and CGOPT (1.0)  in term of the number of iterations,  since it successfully solves about 56$ \% $ test problems with the least iterations, while the numbers of  SMCG$\_$BB and CGOPT (1.0)  are about 37$ \% $ and 27$ \% $, respectively.
				Figure \ref{fig:NfA1CGoptSMCGBB} shows  that Algorithm 1 enjoys large advantage over CGOPT (1.0) and performs slightly better than  SMCG$\_$BB in term of the number of function evaluations.   As shown in Figure \ref{fig:NgA1CGoptSMCGBB}, we can see that Algorithm 1 is superior much to  CGOPT (1.0)  and  SMCG$\_$BB in term of the number of gradient evaluations.  Figure  \ref{fig:TcpuA1CGoptSMCGBB} indicates that Algorithm 1 is faster than  SMCG$\_$BB and CGOPT (1.0).
					
								\begin{figure}[htbp]	
									\begin{minipage}[t]{0.5\linewidth}
										\centering
										\includegraphics[width=9.3cm,height=7cm]{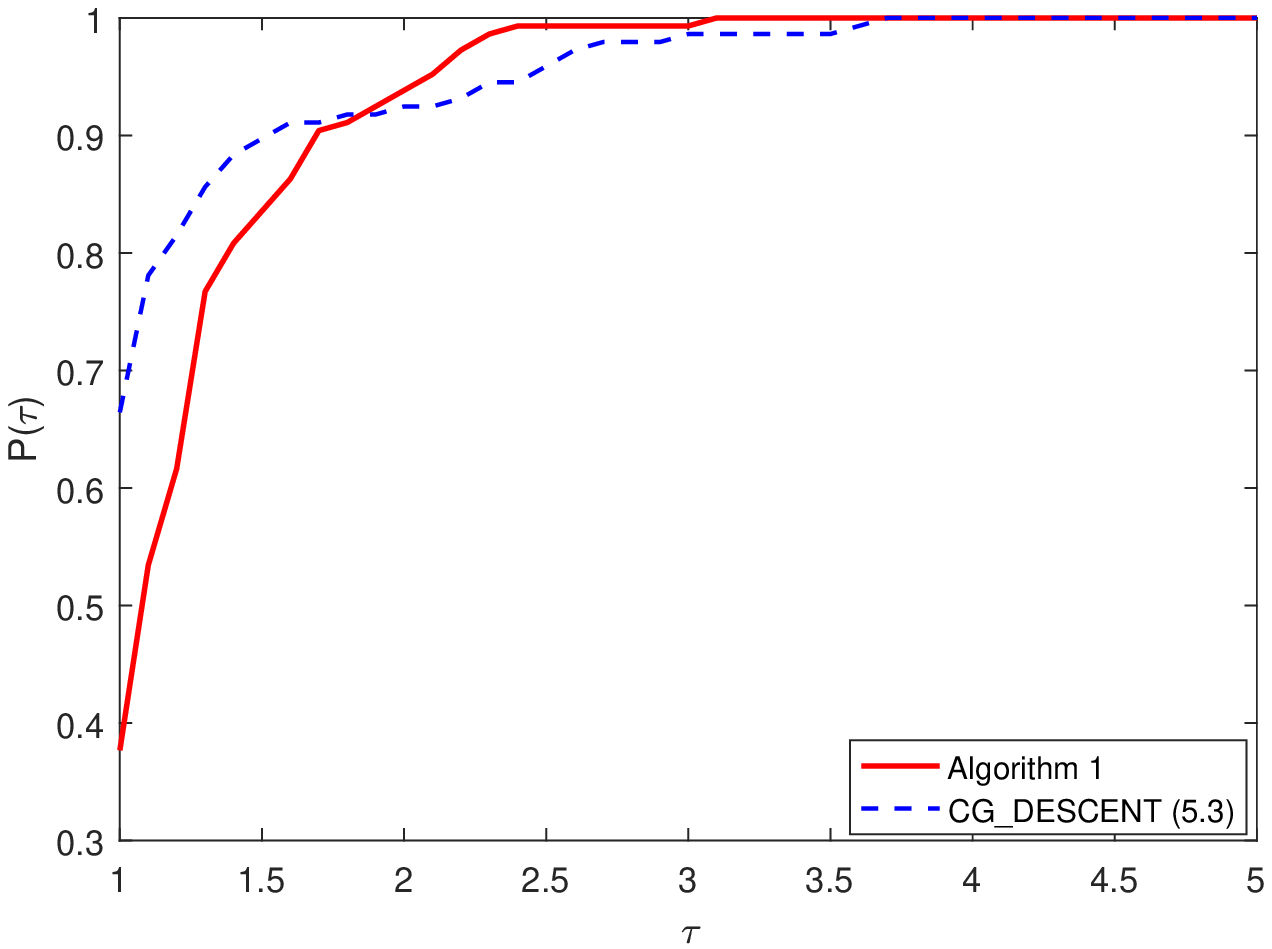}
										\caption{$ N_{f}$  }
										\label{fig:NfA1CG53}
									\end{minipage}%
									\begin{minipage}[t]{0.5\linewidth}
										\centering
										\includegraphics[width=9.3cm,height=7cm]{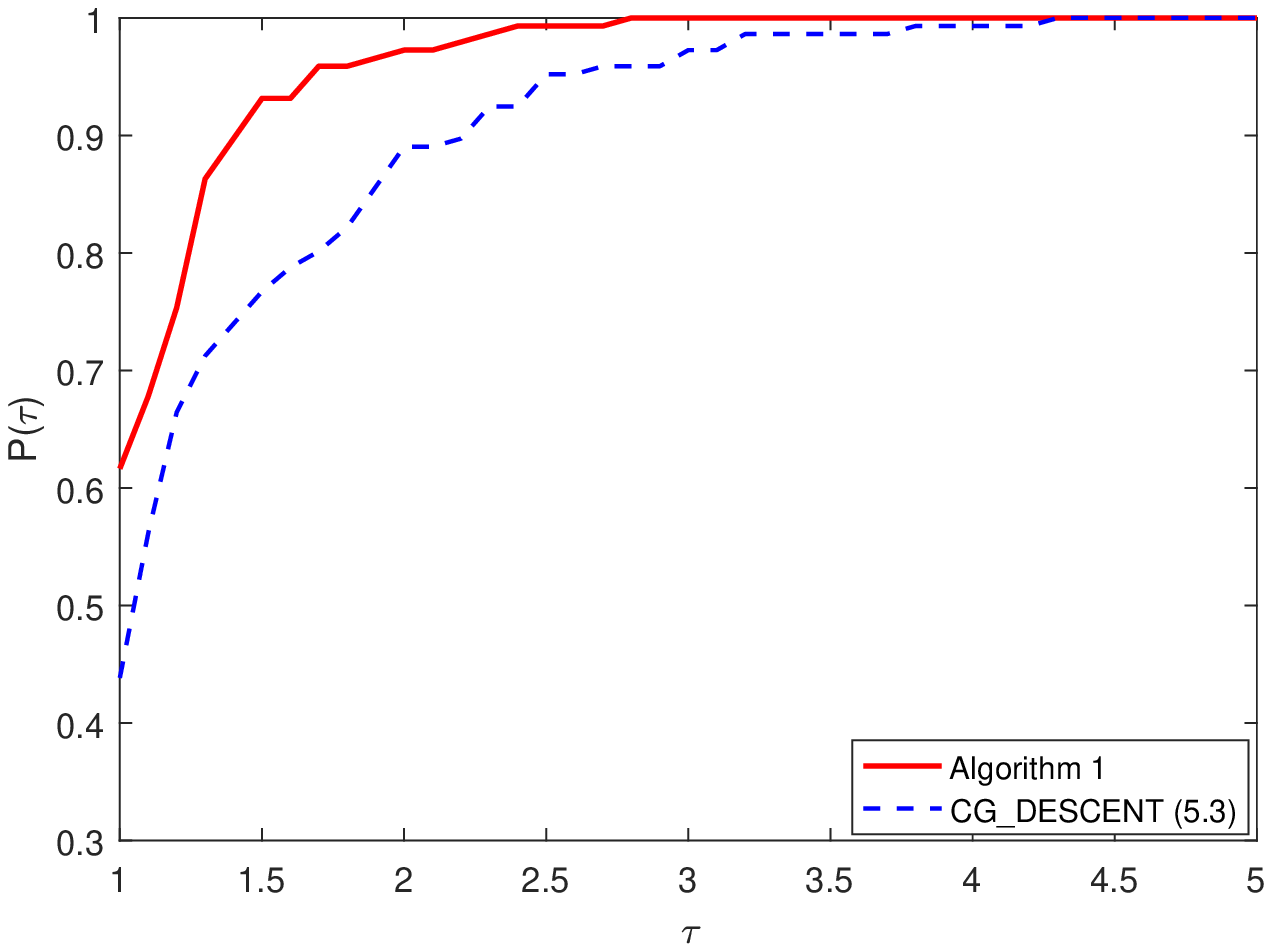}
										\caption{   $ N_{g}$  }
										\label{fig:NgA1CG53}
									\end{minipage}%
								\end{figure}	
								
								\begin{figure}[htbp]	
									\begin{minipage}[t]{0.5\linewidth}
										\centering
										\includegraphics[width=9.3cm,height=7cm]{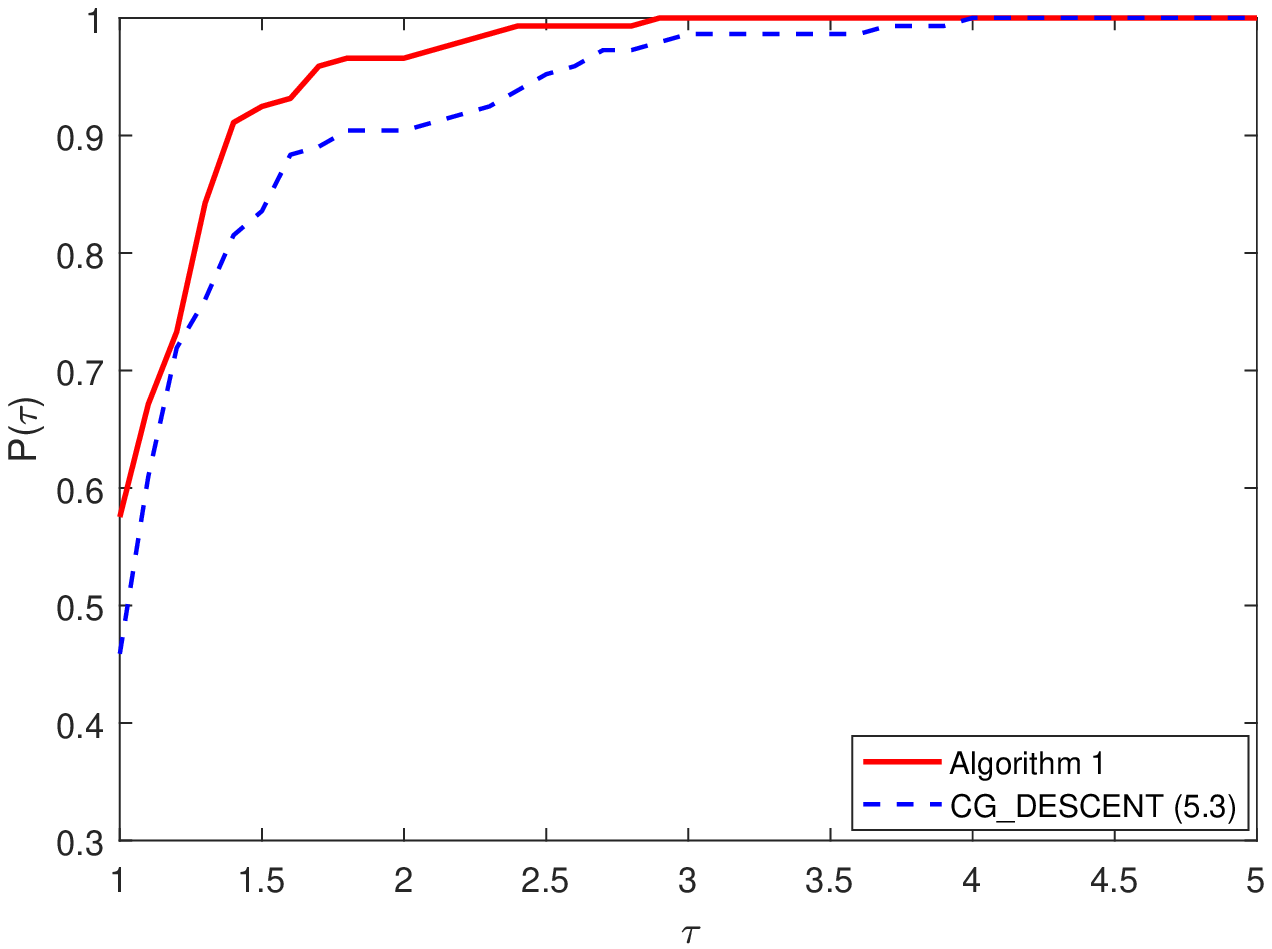}
										\caption{$ N_{f}+3N_g$  }
										\label{fig:Nf+3NgA1CG53}
									\end{minipage}%
									\begin{minipage}[t]{0.5\linewidth}
										\centering
										\includegraphics[width=9.3cm,height=7cm]{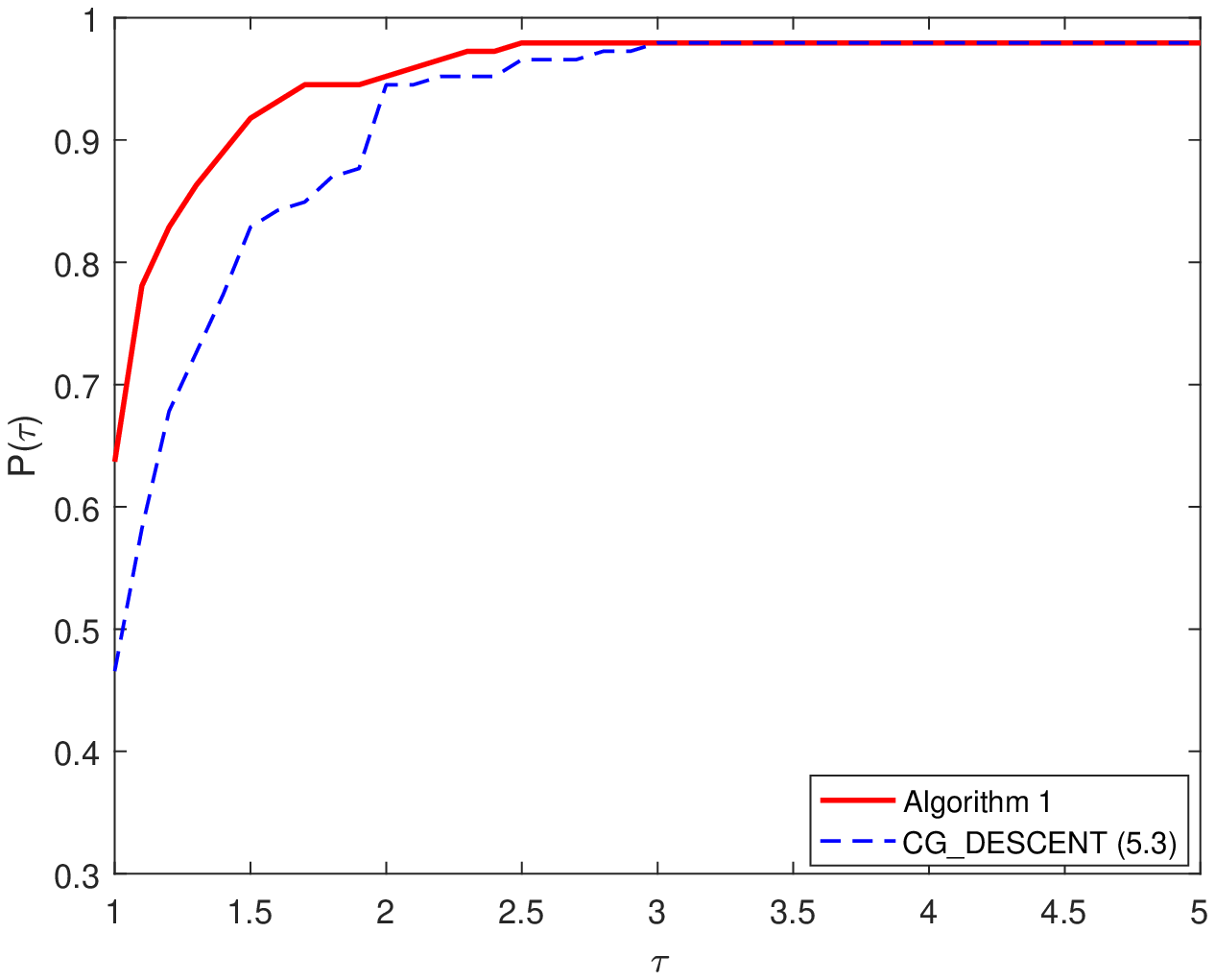}
										\caption{   $ T_{cpu} $  }
										\label{fig:TcpuA1CG53}
									\end{minipage}%
								\end{figure}

				In the third group of the numerical experiments,  we compare the performance of Algorithm 1  with that of  CG$ \_ $DESCENT (5.3). As shown in Figure \ref{fig:IterA1CG53}, we observed that Algorithm 1 performs  better than  CG$ \_ $DESCENT (5.3) in term of the number of iterations, which is a little   beyond our expectations. Figure \ref{fig:NfA1CG53} indicates that  Algorithm 1 is inferior to CG$ \_ $DESCENT (5.3) in term of the number of function evaluations. The reason    is that in the numerical experiments  Algorithm 1   used  the improved Wolfe line search, while CG$\_$DESCENT (5.3) used the combination of the quite efficient  approximate Wolfe line search and the standard  Wolfe line search. It follows from  Section 3 that Algorithm 1 is globally convergent, whereas there is no guarantee for the
				global convergence of CG$\_$DESCENT with the very efficient approximate Wolfe line search \cite{HagerLMCGDESCENT}.  As shown in Figure  \ref{fig:NgA1CG53}, we see that Algorithm 1 enjoys  large advantage  over CG$ \_ $DESCENT (5.3) in term of the number of gradient evaluations.  To see the comprehensive  performance about $N_f$ and $N_g$, we compare the performance  based on  $N_f+3N_g$ in Figure \ref{fig:Nf+3NgA1CG53}.  Figure \ref{fig:Nf+3NgA1CG53} indicates that Algorithm 1  is also superior to    CG$\_$DESCENT (5.3) based on the total performance about $ N_f $ and $ N_g $ though it is at disadvantage in term of $ N_f $. Figure  \ref{fig:TcpuA1CG53} shows that Algorithm 1 is faster than    CG$\_$DESCENT (5.3).

				The above numerical experiments indicate that Algorithm 1 is superior to CGOPT(1.0), CG$ \_ $DESCENT (5.3) and SMCG$ \_ $BB.  It seems that   SMCG method with $ d_k = u_k g_k + v_k s_{k-1} $ can illustrate     greater potential   for  large scale unconstrained optimization compared to the traditional conjugate gradient method with $ d_k =-g_k + \beta_k d_{k-1} $.

					\section{Conclusion and Discussion}
	SMCG methods 		  are quite efficient iterative methods for large scale unconstrained optimization. However, 	 it is usually required to determine the important parameter $\rho_k  \approx g_k^TB_k g_k  $, which is crucial to the   theoretical  properties and the  numerical performance  and is difficult   to be selected  properly. 	 	
					By taking advantage of  the memoryless quasi-Newton method, we present a new subspace  minimization conjugate gradient method  based on  project   technique, which is independent of  the important parameter $ \rho_k $.  It is remarkable that the SMCG method without the exact line search  enjoys finite termination for two dimensional convex quadratic functions, which will guide us  in the design of the proposed method. The proposed method can be regarded as the extension of Dai-Kou conjugate gradient method. The descent property of the search direction is analyzed. We also establish   the global convergence for general nonlinear functions of the proposed method. Numerical experiments indicate that  the proposed method is very promising. We believe that SMCG methods are able to become   strong candidates for large scale unconstrained optimization.					
					

					\begin{acknowledgements} We would like to  thank Professor   Yu-Hong Dai in Chinese Academy of Sciences for his valuable and
						insightful comments on this manuscript.   This research is supported by National Science Foundation of China (No. 12261019, 12161053), Guizhou Provincial Science and Technology Projects (No. QHKJC-ZK[2022]YB084).
					\end{acknowledgements}

\noindent\textbf{Data availability.}					
	The datasets generated during and/or analysed during the current study are available from the corresponding author on reasonable request.	
	
\noindent\textbf{Conflict of interest.	}
	The authors declare no competing interests.

				\end{document}